\documentclass[a4paper,11pt]{article}

\usepackage{anysize}
\marginsize{2cm}{2cm}{2cm}{0cm}

\usepackage[utf8]{inputenc}
\usepackage[scaled]{helvet} 
\usepackage{amsmath,amsthm}
\usepackage{amssymb}
\usepackage{amsfonts}
\usepackage{mdframed}
\usepackage{dsfont}
\usepackage{tikz}
\usepackage{nicefrac}
\usepackage[margin=2.5cm]{geometry}
\usepackage[many]{tcolorbox}
\usepackage{array}
\usepackage{todonotes}
\usepackage{tikz}

\usepackage{graphicx} 
\usepackage{float} 
\usepackage{wrapfig} 
\usepackage[font=small,format=plain,labelfont=bf,up,textfont=it,up]{caption}
\usepackage{subcaption}
\usepackage{mdwlist} 
\usepackage{titlesec} 
\usepackage{multicol}
\usepackage{enumerate}
\usepackage{hyperref}
\usepackage{siunitx}
\titlespacing{\section}{0pt}{8pt}{4pt}
\titlespacing{\subsection}{0pt}{6pt}{2pt}
\usepackage{multicol}
\usepackage[noabbrev]{cleveref}

\def\be{\begin{equation}}
\def\ee{\end{equation}}
\def\bse{\begin{subequations}}
\def\ese{\end{subequations}}
\def\bge{\begin{eqnarray}}
\def\bgee{\begin{eqnarray*}}
\def\ege{\end{eqnarray}}
\def\egee{\end{eqnarray*}}

\usepackage{authblk}

\DeclareMathOperator*{\esssup}{ess\,sup}

\newtheorem{definition}{Definiton}
\newtheorem{theorem}[definition]{Theorem}

\newtheorem{lemma}[definition]{Lemma}

\newtheorem{remark}[definition]{Remark}

\newcommand{\e}{\varepsilon}
\newcommand{\N}{\mathbb{N}}
\newcommand{\R}{\mathbb{R}}

\newcommand{\dive}{\operatorname{div}}
\newcommand{\di}[1]{\,\mathrm{d}#1}

\crefname{in}{inequality}{inequalities}
\crefrangelabelformat{equation}{(#3#1#4--#5#2#6)}
\creflabelformat{in}{#2{\upshape(#1)}#3} 
\crefname{assumption}{Assumption}{assumptions}

\newtcolorbox{mybox}[1]{%
    tikznode boxed title,
    enhanced,
    arc=0mm,
    interior style={white},
    attach boxed title to top left= {yshift=-\tcboxedtitleheight/2-0.05cm, xshift=0.7cm},
    fonttitle=\small\bfseries,
    colbacktitle=white,coltitle=black,
    boxed title style={size=small,colframe=white,boxrule=0pt},
    title={#1}}

\makeatletter
\newcommand{\tline}{%
    \noalign {\ifnum 0=`}\fi \hrule height 1pt
    \futurelet \reserved@a \@xhline
}
\newcolumntype{"}{@{\hskip\tabcolsep\vrule width 1pt\hskip\tabcolsep}}
\makeatother
\begin{document}
\title{A multiscale quasilinear system for colloids deposition in porous media: Weak solvability and numerical simulation of a near-clogging scenario}

\author[$\dagger$]{Michael Eden}
\author[$\star$]{Christos Nikolopoulos}
\author[$\ddag$]{Adrian Muntean}

\affil[$\dagger$]{Zentrum f\"ur Technomathematik, Department of Mathematics and Computer Science, University of Bremen, Germany}
\affil[$\star$]{Department of Mathematics, School of Sciences, University of The Aegean, Greece}
\affil[$\ddag$]{Department of Mathematics and Computer Science, Karlstad University, Sweden}

\maketitle
\begin{abstract}
 We study the weak solvability of a quasilinear reaction-diffusion system nonlinearly coupled with an linear elliptic system posed in a domain with distributed microscopic balls in $2D$.
 The size of these balls are governed by an ODE with direct feedback on the overall problem. The system describes the diffusion, aggregation, fragmentation, and deposition of populations of colloidal particles of various sizes inside a porous media made of prescribed  arrangement of balls.
 The mathematical analysis of the problem relies on a suitable application of Schauder's fixed point theorem which also provides a convergent algorithm for an iteration method to compute finite difference approximations of smooth solutions to our multiscale model.
 Numerical simulations illustrate the behavior of the local concentration of the colloidal populations close to clogging situations.
\end{abstract}




\section{Introduction and problem statement}

We study a two-scale system modeling the effective diffusive transport as well as the aggregation, fragmentation, and deposition of populations of colloidal particles inside porous media.
Such situations arise, for instance, in membrane filtration scenarios \cite{Fasano,Bruna_JFM}, papermaking \cite{Asa},  immobilization of colloids in soils \cite{Chen}, or transport of colloidal contaminants in groundwater \cite{Suciu}.

We are particularly interested in situations where micro-structural changes due to the deposition or dissolution of colloids are allowed to take place.
This can locally change both the transport patterns and storage capacity of the medium; see \cite{Icardi,King,Hallak,Maes,Knabner,Noorden} for related cases. This variety of technological and natural processes is based on the transfer of colloidal particles from liquid suspension onto stationary surfaces \cite{johnson1995dynamics}.
From this perspective, one can perceive that the porous media we are considering here behave like materials with reactive internal microstructures (see \cite{Diaz} for a periodic setting) and, based on \cite{Showalter_Oberwolfach}, they are sometimes classified as media with distributed microstructures.
Additional motivation for this work comes  from our own research on reactive flow in porous media and is linked very much with the work of P. Ortoleva and J. Chadam (see e.g. \cite{Chadam} and follow up papers), but it is worth mentioning that quite related  aspects arise in pharmacy and medicine like drug delivery, thrombosis formation on arterial walls, evolution of Alzheimer's disease.
We refer the reader, for instance, to  \cite{Giulia,Thrombosis,Silvia} for works in this direction.

Denoting with $u=(u_1,...,u_N)$ ($i=1,...,N$) the molar concentrations of colloids of size $i$ (with $N\in \N$ the maximal size), its time evolution can be modelled by a quasi-linear parabolic system in the form of

\begin{align}\label{abstract_quasi}
\partial_tu_i-\dive(D_i(u)\nabla u_i)=F_i(u),
\end{align}
where $F_i(u)$ accounts for the aggregation, segregation, and adsorption processes and $D_i(u_i)$ the changing permeability as consequence of the micro-structural changes (like clogging) inside the porous medium itself.
While \cref{abstract_quasi} is purely macroscopic, the computation of the effective permeability $D_i(u_i)$ is done on the micro-scale therefore leading to the two-scale nature of our problem.
This system is a compact and abstract reformulation of a two-scale model for colloidal transport derived in \cite{MC20} via asymptotic homogenization (more details are given in \Cref{strategy}).
Structurally similar (two-scale model with geometrical changes) models were investigated in, e.g., \cite{Eden19,Peter09}.

In this work, we take a $2D$-cross-section of a porous medium and assume the solid matrix of the cross section to be made up of circles of not-necessarily uniform radius.
The growth and shrinkage of these circles, which represent the underlying micro-structural changes of the porous medium, are modelled via a scalar quantity governed by an additional ODE.
For a similar geometrical setup see, e.g., \cite{Peter09}.
The model and the resulting mathematical problem gets more complicated if we were to allow for more general geometries (e.g., evolving $C^2$-interfaces) that can not be represented by a scalar quantity like the radius in our setting.
We treat our geometries in $2D$ mainly for the sake of simplicity of inequalities and transformations and also because the simulation work is easier to be handled in $2D$ compared to $3D$, there is no fundamental element in the analysis that is sensitive to dimensions (like Sobolev embeddings would be for example).
As a consequence, the mathematical analysis part can be extended to $3D$ with suitable  modifications on the upper and lower {\em a priori} bounds on the radii of the balls-like microstructure.

The quasilinear structure of the problem together with the multiscale coupling is non-standard.
Here, we point out that $D_i$ and $F_i$ are non-linear operators that are not defined via point wise evaluation (in the sense of $D_i(u)(t,x)=D_i(t,x,u(t,x))$).
In particular, it does not fit directly to the framework elaborated in, e.g., \cite{Alt} and it requires an approach that utilizes the underlying coupling present in the model equations behind the abstract system.
A similar two-scale problem allowing for micro-structural changes was investigated in \cite{Meier09}.

In \Cref{strategy} we explain our working strategy to prove the existence of weak solutions to the overall problem.
To keep things simple, we consider that the local porosity $\phi(r)$ does not degenerate.
Note however, that it is technically possible to include in the analysis simple degeneracies (like neighboring microstructures touching in single points \cite{Schulz}), a complete (local) clogging being however out of reach. 
Besides the non-degeneracy of the effective parameter, another simplification is included  -- the absence of the flow.
Note that if the colloidal populations would be immersed in a fluid flow, then, most likely, besides the balance equations of the linear momentum one would also have to take into account the charge transport taking place between oppositely charged populations of particles; see e.g. \cite{Robin,Ray} for more information in this direction.     

The paper is organized as follows: In \Cref{strategy} we present the model and outline our strategy for the analysis of our problem.
We list the needed mathematical details of the problem so that we can prove in \Cref{existence} the existence of a weak solution.
In \Cref{numerics},  we solve numerically our multiscale quasilinear problem and discuss the obtained numerical results for realistic parameter regimes.
We add in \Cref{discussion} a detailed discussion of the potential of our problem, expected results, and related aspects. 

\section{Problem statement and solution strategy}\label{strategy}
In the following, let $S=(0,T)$ be the time interval of interest and $\Omega\subset\R^2$ a bounded Lipschitz domain.
In addition, let $N\in\N$ be a given number indicating the maximal possible \emph{size} of an aggregate of colloid particle, where \emph{size} refers to the number of primary particles making up the aggregate.
For each $i=\{1,...,N\}$, let $u_i\colon S\times\Omega\to[0,\infty)$ (we set $u=(u_1,...,u_N)$) denote the molar concentration density of aggregates of size $i$ at point $x\in\Omega$ at time $t\in S$.
We take the function $v\colon S\times\Omega\to[0,\infty)$ to represent the mass density of absorbed material (mass that is in the system but currently not part of the diffusion and agglomeration process); this mass can be dissolved again by a  Robin-type exchange allowing colloidal populations to re-enter the pore space. 
This process of absorption and dissolution is modelled in this context  via an Robin-type exchange term (see e.g. \cite{Krehel}) in the form of

$$
\frac{2\pi r}{1-\pi r^2}(a_iu_i-\beta_iv).
$$
Here, the radius function $r\colon S\times\Omega\to(0,r_{max})$ (for some $r_{max}>0$)  acts as a measure of the \emph{clogginess} of the porous media and $\frac{2\pi r}{1-\pi r^2}$ is the ratio of the size of the boundary between fluid space and pore to the pore volume.

To describe  the aggregation and fragmentation processes taking place inside the pore space of the medium, we use the \emph{Smoluchowski} formulation (we point to \cite{Aldous} for a review) given here by

%
$$
R_i(u)=\frac{1}{2}\sum_{j+l=i}\gamma_{jl}u_ju_l-u_i\sum_{j=1}^{N-i}\gamma_{jl}u_j.
$$
It is important to note that in the context of porous media the colloidal populations involve a finite size chain of the cluster, i.e. there will be a population of $N$-mers where $N$ takes the maximum cluster size.
As a result, we deal with a finite sum here.
Interestingly, for many applications a good choice of such $N$ is rather low; see e.g.  \cite{Krehel}.

The diffusion-reaction system for the different aggregates is then given via
\begin{subequations}
\begin{alignat}{2}
\partial_tu_i-\dive(D_i(r)\nabla u_i)&=R_i(u)-\frac{2\pi r}{1-\pi r^2}(a_iu_i-\beta_iv)&\quad&\text{in}\ \ S\times\Omega,\label{overall-1}\\
-D_i(r)\nabla u_i\cdot n&=0&\quad&\text{on}\ \ S\times\partial\Omega,\label{overall-2}\\
u_i(0)&=u_{i0}&\quad&\text{in}\ \ \Omega.\label{overall-3}
\end{alignat}

The effective diffusion matrix (including diffusion, dispersion, and tortuosity effects) $D_i(r)\in\R^{2\times2}$ can be calculated using any solution $w_{k}$, $k=1,2$, of the cell problem
\begin{alignat}{2}
-\Delta w_{k}&=0&\quad&\text{in}\ \ S\times(Y\setminus \overline{B}(r)),\label{overall-4}\\
-\nabla w_{k}\cdot n&=e_k\cdot n&\quad&\text{on}\ \ S\times\partial B(r),\label{overall-5}\\
y&\mapsto w(\cdot,\cdot,y)&\quad&\text{is $Y$-periodic}\label{overall-6}.
\end{alignat}
Here, $Y=(0,1)^2$ denotes the unit cell, $\overline{B}(r)$ is the closed ball with radius $r$ and center point $a=(\nicefrac{1}{2},\nicefrac{1}{2})$, and $e_k$ the $k$-th unit normal vector. 
We have ($d_i>0$ are known constants)
$$
(D_i)_{jk}=d_i\phi(r)\int_{Y\setminus \overline{B}(r)}(\nabla w_{k}+e_{k})\cdot e_j\di{z}
$$
where $\phi(r)=\frac{1-\pi r^2}{|\Omega|}$ denotes the porosity density of the medium.
For more details regarding the cell problem and the effective diffusivity, we refer to \cite{MC20} where they are established via homogenization.

Finally, the evolution of $v$ is governed by an ODE parametrized in $x\in\Omega$
\begin{alignat}{2}
\partial_tv&=\sum_{i=1}^N\left(\alpha_iu_i-\beta_iv\right)&\quad&\text{in}\ \ S\times\Omega,\label{overall-7}\\
v(0)&=v_0&\quad&\text{in}\ \ \Omega.\label{overall-8}
\end{alignat}
and the radius function is governed by the following ODE parametrized in $x\in\Omega$
\begin{alignat}{2}
\partial_tr&=2\pi\alpha\sum_{i=1}^N\left(a_iu_i-\beta_iv)\right)&\quad&\text{in}\ \ S\times\Omega,\label{overall-9}\\
r(0)&=r_0&\quad&\text{in}\ \ \Omega.\label{overall-10}
\end{alignat}
\end{subequations}
A possible initial choice for the radii $r_0$ is depicted in Figure \ref{Rdunit}. We point out there also what will happen at the final time $T$; more details on the parameter setup are given in the simulation sections.   What concerns the modeling of the deposition of the colloidal populations, our choice is similar to one reported in \cite{johnson1995dynamics}. 

\begin{figure}[h!]
\centering
\includegraphics[trim={5cm 10cm 4.5cm 10cm},clip,scale=.68]{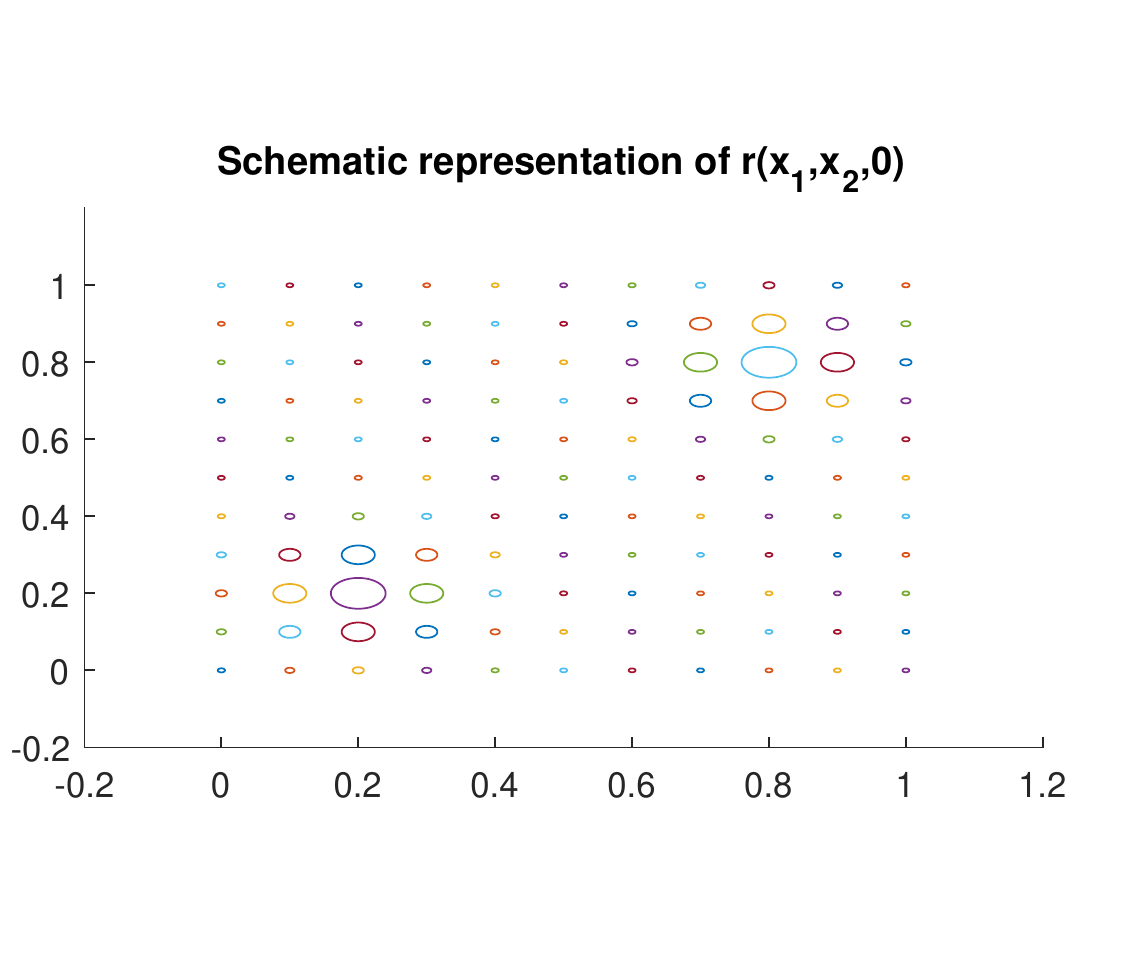}
\includegraphics[trim={5cm 10cm 4.5cm 10cm},clip,scale=.68]{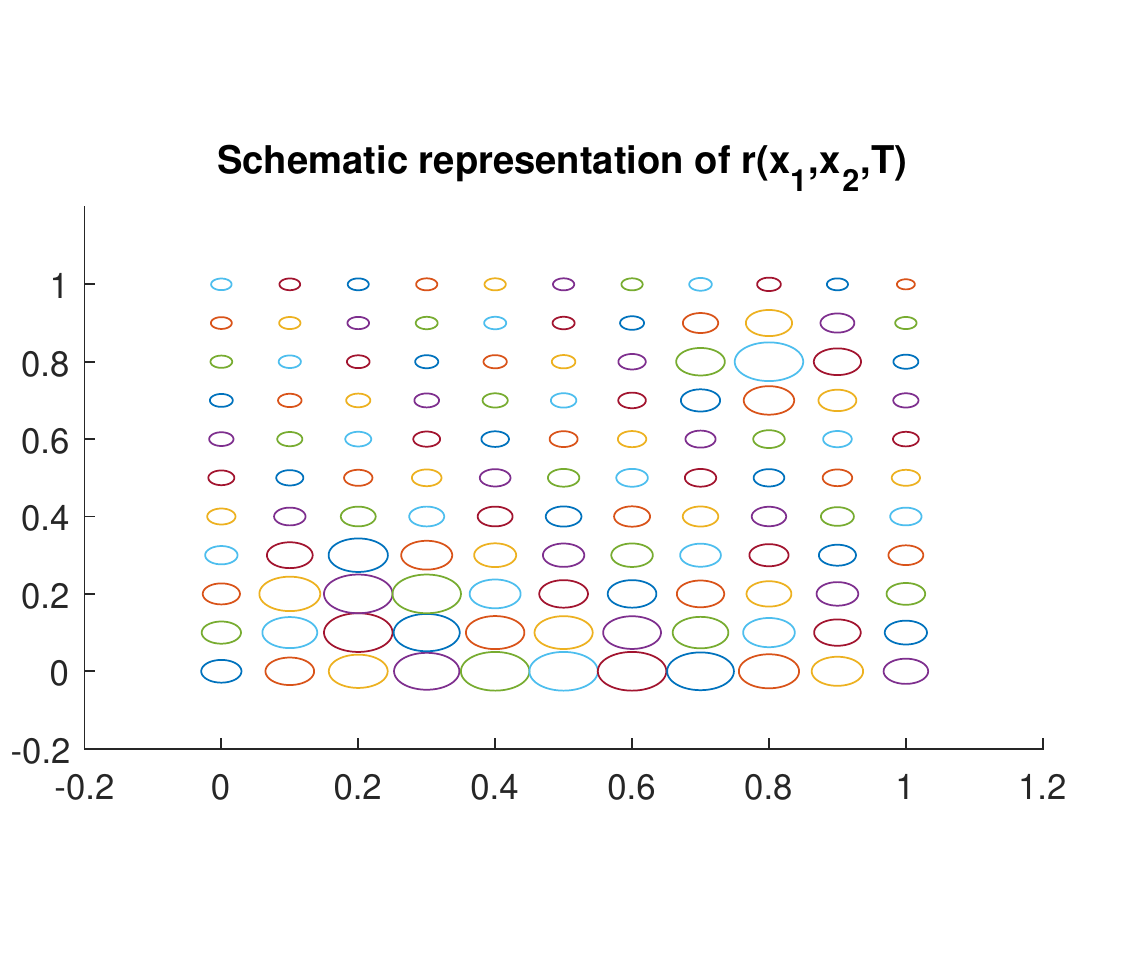}
 \caption{Example of $r(x_1,x_2,t=0)$ with corresponding $r(x_1,x_2,t=T)$ of the same simulation.
 The parameter setting is as discussed in Figure \ref{Figex2D}.
 Regions with larger circles correspond to low porosity and permeability.} 
 \label{Rdunit}
\end{figure} 
This accounts for the simple observation that the absorbed material leads to the clogging of the pore under the fundamental assumption of the growth of the radius is proportional to the amount of material that is absorbed.
For a more concrete argumentation for this particular structure, we again point to \cite{MC20}.

The overall problem we are considering in this work is then given by \cref{overall-1,overall-2,overall-3,overall-4,overall-5,overall-6,overall-7,overall-8,overall-9,overall-10}.
Regarding our concept of a weak solution of this system:

\begin{definition}[Weak solution]\label{weaksol}
For a time interval $(0,s)\subset S$, a weak solution to the problem is given by a set of functions $(u,v,w,r)$ with the regularity
\begin{align*}
u_i&\in L^2((0,s);H^1(\Omega))\cap L^\infty((0,s)\times\Omega)\quad\text{such that}\ \ \partial_tu_i\in L^2((0,s)\times\Omega),\\
w&\in L^2((0,s)\times\Omega;H^1_\#(Y)),\quad
v\in W^{1,1}((0,s);L^2(\Omega)),\quad r\in W^{1,1}((0,s);L^2(\Omega))
\end{align*}
that satisfies \cref{overall-1,overall-2,overall-3,overall-4,overall-5,overall-6,overall-7,overall-8,overall-9,overall-10} in the standard weak Sobolev setting.
\end{definition}

\paragraph{Solution strategy.} Without yet caring  about regularity issues (like smoothness, integrability, measurability) and possible singularities, we want to suggest our solution strategy for the problem given by \cref{overall-1,overall-2,overall-3,overall-4,overall-5,overall-6,overall-7,overall-8,overall-9,overall-10} and show how it relates to the abstract quasi-linear PDE System \ref{abstract_quasi}.

We start with a few comments regarding the particular structure of our problem where we refer to the subproblems $(i)$-$(iv)$ for $u, w, v, r$, viz.

\begin{enumerate}
	\item[(A)] The problem is strongly coupled: $(i)$ depends on $u, w, v, r$, $(ii)$ on $w, r$, $(iii)$ on $u, v$, and $(iv)$ on $r, u, v$.
	\item[(B)] Problem $(i)$ is parabolic in $u$, $(ii)$ elliptic in $w$, $(iii)$ and $(iv)$ are first order ODEs in $v$ and $r$.
	\item[(C)] Problem $(i)$ is nonlinear in $u$ and $r$, $(ii)$ is nonlinear in $r$, and $(iii)$ and $(iv)$ are linear.
	\item[(D)] Problem $(ii)$ is not a \emph{real} free boundary problem, as the underlying domain $Y\setminus\overline{B(r)}$ depends only on $(t,x)$ while the derivatives are w.r.t.~$y$.
\end{enumerate}
As a consequence of points (A)--(D), a natural strategy is to first tackle the ODEs and to use them to inform the cell problem and the parabolic system.
In the following, we outline the intermediate steps involved in getting to the abstract fixed-point problem that will be the starting point for our analysis in \Cref{existence}:

Step $(a)$: Looking at the linear ODE vor $v$ (given by \cref{overall-7,overall-8}),
				we find the characterization of $v$ in terms of $u$ via (setting $b=\sum_{i=1}^N\beta_i$)
				$$
				v(t,x)=e^{-bt}\left(v_0(x)+\sum_{i=1}^N\alpha_i\int_0^te^{b\tau}u_i(\tau,x)\di{\tau}\right).
				$$
				With this in mind, we can eliminate $v$ for $u$ in our problem by setting $v=\mathcal{L}_v(u)$, where $\mathcal{L}_v$ is the abstract solution operator for the $v$-problem.

Step $(b)$: Similarly, looking at the second ODE (problem $(iv)$), we have
    $$
    r(t,x)=r_0(x)+2\pi\alpha\sum_{i=1}^N\int_0^t(a_iu_i(\tau,x)-\beta_iv(\tau,x))\di{\tau}
    $$
		With this characterization, we can introduce the corresponding solution operator $\widetilde{\mathcal{L}_{r}}$ via
		$$
		r=\mathcal{L}_{r}(u,v)=\mathcal{L}_{r}(u,\mathcal{L}_{v}(u))=\widetilde{\mathcal{L}_{r}}(u).
		$$

Step $(c)$: Looking at the cell problem $(k=1,2)$
		\begin{alignat*}{2}
		-\Delta w_{k}&=0&\quad&\text{in}\ \ S\times(Y\setminus \overline{B}(r)),\\
		-\nabla w_{k}\cdot n&=e_k\cdot n&\quad&\text{on}\ \ S\times\partial B(r),\\
		y&\mapsto \tau(\cdot,\cdot,y)&\quad&\text{is $Y$-periodic},
		\end{alignat*}
		we expect to get solutions for every given $r>0$ such that $\overline{B}(r)\cap\partial Y=\emptyset$.
		We introduce the corresponding solution operator via
		$$w=\mathcal{L}_{w}(r)=\left(\mathcal{L}_{w}\circ \widetilde{\mathcal{L}_{v}}\right)(u)=\widetilde{\mathcal{L}_{w}}(u).$$
	
Step $(d)$: Putting everything together, we can rewrite the parabolic problem
		$$
		\partial_tu_i-\dive(D_i(r,w)\nabla u_i)=R_i(u)-\frac{2\pi r}{1-\pi r^2}(a_iu_i-\beta_iv)
		$$
		into
		$$
		\partial_tu_i-\dive\left(D_i\left(\widetilde{\mathcal{L}_{r}}(u),\widetilde{\mathcal{L}_{w}}(u)\right)\nabla u_i\right)=R_i(u)-\frac{2\pi\widetilde{\mathcal{L}_{r}}(u)}{1-\pi (\widetilde{\mathcal{L}_{r}}(u))^2}(a_iu_i-\beta_i\mathcal{L}_{v}(u))
		$$
		This highly nonlinear system of PDEs is now given only in terms of the unknown function $u$.
		On an abstract level, we therefore want to investigate parabolic system like
		\begin{subequations}
    		\begin{alignat}{2}
        		\partial_tu_i-\dive\left(\widehat{D_i}(u)\nabla u_i\right)&=F_i(u)&\quad&\text{in}\ \ S\times\Omega,\label{nonlinear-1}\\
        		-\widehat{D_i}(u)\nabla u_i\cdot n&=0&\quad&\text{on}\ \ S\times\partial\Omega,\label{nonlinear-2}\\
        		u_i(0)&=u_{i0}&\quad&\text{in}\ \ \Omega\label{nonlinear-3}
    		\end{alignat}
		\end{subequations}
		where
		$$
		F_i(u)=R_i(u)-\frac{2\pi\widetilde{\mathcal{L}_{r}}(u)}{1-\pi (\widetilde{\mathcal{L}_{r}}(u))^2}(a_iu_i-\beta_i\mathcal{L}_{v}(u)).
		$$
The exact setting regarding function spaces will be settled in the following section.


\section{Analysis}\label{existence}
In this section, we present the detailed fixed-point argument (as outlined in Section \ref{strategy}) for the non-linear problem given via \cref{nonlinear-1,nonlinear-2,nonlinear-3}:






The strategy of our proof is a three-step process:

\begin{enumerate}
    \item[1)] For a given function $\tilde{u}$ (of sufficient regularity), we establish well-posedness  and estimates for the  linear problem given by 
    \begin{subequations}
    \begin{alignat}{2}
    	\partial_tu_i-\dive\left(\widehat{D_i}(\tilde{u})\nabla u_i\right)&=F_i(\tilde{u})&\quad&\text{in}\ \ S\times\Omega,\label{linearized-1}\\
    	-\widehat{D_i}(\tilde{u})\nabla u_i\cdot n&=0&\quad&\text{on}\label{linearized-2}\ \ S\times\partial\Omega,\\
    	u_i(0)&=u_{i0}&\quad&\text{in}\ \ \Omega\label{linearized-3}.
    \end{alignat}
    \end{subequations}
    This is established in \Cref{existence_linear}.
    \item[2)] We show that there is a set such that the solution operator for \cref{linearized-1,linearized-2,linearized-3} maps that set onto itself, see \Cref{lemma_fixed}.
    This result is local in time, since we need to keep $t$ small in order to control the norm of the solution.
    \item[3)] Finally, we employ Schauder's fixed point theorem to establish the existence of at least one solution, see \Cref{existence}.
\end{enumerate}

For some arbitrary (later to be fixed) $M>0$ and $s\in(0,T)$, let

$$
T_{s,M}=\{u\in L^2((0,s)\times\Omega)^N\ : \ \|u_i\|_\infty\leq M \ (i=1,...,N)\}.
$$
For ease of notation, for any given $u$ of sufficient regularity we will write $v_u=\mathcal{L}_v(u)$, $r_u=\widetilde{\mathcal{L}_r}(u)$, $w_u=\widetilde{\mathcal{L}_w}(u)$ for the corresponding solution given for the particular subproblem and $Y_u=Y\setminus\overline{B(r_u)}$.

\subsection{Auxiliary results}
We start by collecting some important auxiliary results and estimates that will be needed in the construction of the actual fixed-point argument.

\begin{table}[h]
\centering
\begin{tabular}{l"l"l"}
Function						& Assumption 										& Reason	 \\ \tline
$r_0$								& $\nicefrac{1}{8}\leq r_0(x)\leq\nicefrac{1}{4}$	& Room for growth and shrinkage\\
$u_{i0}$						& $0\leq u_{i0}(x)\leq\nicefrac{M}{2}$				& Keeping the solution in $T_{s,M}$    \\ 
$v_0$								& $0\leq v_0(x)\leq C_v$	& Bounding $v_u$
\end{tabular}
\caption{Assumptions regarding the initial data.}
\end{table}

In a first step, we establish some sufficient conditions for the diffusivity matrix to not degenerate.
Note that at this point it is not clear that this condition can be satisfied; this is shown in \Cref{lemma_boundsr}.
\begin{lemma}[Diffusivity]\label{lemma_diffus}
If $u\in L^2((0,s)\times\Omega)$ is chosen such that $0\leq2r_u\leq 1-\e_1$ for some small $\e_1>0$, we find that $\widehat{D_i}(u)$ is symmetric and positive definite, i.e., $\widehat{D_i}(u)\xi\cdot\xi\geq c_i|\xi|^2$ where the constants $c_i>0$ do not depend on $u$ and $\xi\in\R^2$.
In addition, $\widehat{D_i}(u)\in L^\infty((0,s)\times\Omega)$ .
\end{lemma}
\begin{proof}
Its entries are given by

$$
(\widehat{D_i}(u))_{jk}=d_i\phi(r_u)\int_{Y_u}(\nabla w_{u,k}+e_{k})\cdot e_j\di{z}
$$
where $r_u=\widetilde{\mathcal{L}_{iv}}(u)$, $Y_u=Y\setminus\overline{B}(r_u)$, and $w_u=(w_{u,1},w_{u,2})=\mathcal{L}_{ii}(r_u)$.
The $D_i$ are symmetric since
$$
\int_{Y_u}(\nabla w_{u,k}+e_{k})\cdot e_j\di{z}=\int_{Y_u}(\nabla w_{u,k}+e_{k})\cdot\left(\nabla w_{u,j}+e_j\right)\di{z}
$$
by way of $w_{u,k}$ solving the cell problem.

Via that representation, non negativity is also straightforward to show (we refer to \cite[Section 12.5]{PS08} for a similar argument) as long as $\phi(r_u)$ is non negative.
For the positivity, we have to ensure that there is some $c_i>0$ such that $\phi(r_u),\, |Y_u|\geq c_i$ for all $(t,x)\in S\times\Omega$.
Both hold true if $r_u$ is bounded away from $\nicefrac{1}{2}$, i.e, if there is some $\e_1>0$ such that $2r_u\leq1-\e_1$ for all $(t,x)\in S\times\Omega$.

Now, regarding the boundedness of $D_i$, we first see that $|\phi(r_u)|\leq|\Omega|^{-1}$ when $0\leq2r_u\leq1-\e_1$ is satisfied.
Due to $|Y_u|\leq|Y|=1$, boundedness of $D_i$ is clear.
\end{proof}

In the following, we will try to establish sufficient conditions for a function $u\in L^2((0,s)\times\Omega)$ to guarantee that the condition $2r_u\leq1-\e_1$ is met.
Setting

$$
	a_u(t,x)=2\pi\alpha\sum_{i=1}^N(a_iu_i-\beta_iv_u),
$$
we get

\begin{equation}
	r_u(t,x)=r_0(x)+\int_0^ta_u(\tau,x)\di{\tau}.
\end{equation}

\begin{lemma}[Bounds for $r$]\label{lemma_boundsr}
If $M,\e_1,\e_2>0$ satisfy
\begin{equation}\label{eq:lemma_boundsr}
    M\left(e^{bt}-1\right)\leq \frac{b}{2\pi\alpha a}\min\{1-2\sup r_0-\e_1, \inf r_0-\e_2-2\pi\alpha bt\sup v_0\}
\end{equation}
for all $t\in(0,s)$, it holds  $2\e_2\leq 2r_u\leq1-\e_1$ for all $u\in T_{s,M}$.
\end{lemma}
\begin{proof}
For every $u\in T_{s,M}$, we find that

$$
v_u(t,x)=e^{-bt}\left(v_0(x)+\sum_{i=1}^Na_i\int_0^te^{b\tau}u_i(\tau,x)\di{\tau}\right).
$$
As a consequence,

$$
-\frac{a}{b} M(e^{bt}-1)\leq v_u(t,x)\leq v_0(x)+\frac{a}{b} M(e^{bt}-1).
$$
This implies

$$
a_u=2\pi\alpha\sum_{i=1}^N(a_iu_i-\beta_iv_u)\leq2\pi\alpha\left( aM+aM(e^{bt}-1)\right)=2\pi\alpha aMe^{bt}
$$
as well as

$$
a_u\geq-2\pi\alpha\left(aM+bv_0(x)+aM(e^{bt}-1)\right)=-2\pi\alpha\left(bv_0(x)+aMe^{bt}\right).
$$
Therefore,

$$
\inf r_0-2\pi\alpha\left(tb\sup v_0+\frac{aM}{b}(e^{bt}-1)\right)\leq r_u(t,x)\leq \sup r_0+2\pi\frac{\alpha aM}{b}(e^{bt}-1).
$$
As a consequence, $2\e_2<2r_u<1-\e_1$ can be ensured by the following two conditions:

\begin{align*}
M\left(e^{bt}-1\right)&\leq\frac{b}{2\pi\alpha a}\left(1-2\sup r_0-\e_1\right),\\
M\left(e^{bt}-1\right)&\leq \frac{b}{2\pi\alpha a}(\inf r_0-\e_2-2\pi\alpha bt\sup v_0).
\end{align*}
\end{proof}

\begin{remark}
The condition \ref{eq:lemma_boundsr} required in \Cref{lemma_boundsr} can always be met (over some possibly small time interval $(0,s)$) for $M, \e_1, \e_2$ small enough as long as the initial radius distribution satisfies $2\e_2<2r_0(x)<1-\e_1$.
Connecting \Cref{lemma_boundsr} with \Cref{lemma_diffus} leads to well behaved diffusivities for $u\in T_{s,M}$.
The additional bound from below in the form of $\e_2$ is needed for the transformation for the cell problem for $w_k$.
\end{remark}

Now, looking at the r.h.s.~of our reaction diffusion equation, we have for $u\in T_{s,M}$ (setting $\gamma=\max_{i,j}\gamma_{ij}$):

\begin{equation}\label{est_F1}
-M^2\gamma\left(N-\frac{k+1}{2}\right)\leq R_k(u)\leq M^2\gamma\left(N-\frac{k+1}{2}\right)\quad (1\leq k\leq N).
\end{equation}
Due to $r_u\leq\nicefrac{1}{2}$ and

$$
\frac{2\pi r_u}{1-\pi r_u^2}\leq\frac{\pi}{1-\nicefrac{\pi}{4}}\leq15
$$
we arrive at

\begin{equation}\label{est_F2}
\frac{2\pi r_u}{1-\pi r_u^2}(a_iu_i-\beta_iv_u)\leq 15\left(a_iM+\frac{a}{b}\beta_iM(e^{bt}-1)\right),
\end{equation}
and

\begin{equation}\label{est_F3}
\frac{2\pi r_u}{1-\pi r_u^2}(a_iu_i-\beta_iv_u)\geq-15\left(a_iM+\beta_i\left(v_0(x)+\frac{a}{b}M(e^{bt}-1)\right)\right).
\end{equation}
As a consequence, for every $u\in T_{s,M}$, we find that $F_i(u)\in L^\infty(S\times\Omega)$ for all $i=1,...,N$.
In particular, we find that

\begin{equation}\label{eq:rhs}
\sup\{\|F_i(u)\|_\infty\ : u\in T_{s,M}\}=C
\end{equation}
where the constant $C$ depends only $s,M$.
\begin{lemma}[Estimates for the radius]\label{lem_rad}
For $u^{(1)},u^{(2)}\in T_{s,M}$ let $r^{(1)}, r^{(2)}$ be the corresponding solutions of the radius ODE problem.
Then,

\begin{align*}
\left|r^{(1)}-r^{(2)}\right|
&\leq C\int_0^t\left(\left|u^{(1)}-u^{(2)}\right|+\int_0^\tau e^{bs}\left|u^{(1)}-u^{(2)}\right|\di{s}\right)\di{\tau}.
\end{align*}
where the constant $C>0$ is independent of the particular choice of $u^{(k)}$ ($k=1,2$)
\end{lemma}
\begin{proof}
The radius ODE can be solved by integration ($k=1,2$):

$$
r^{(k)}(t,x)=r_0(x)+2\pi\alpha\sum_{i=1}^N\int_0^ta_iu_i^{(k)}(\tau,x)-\beta_iv^{(k)}(\tau,x)\di{\tau}
$$
where $v^{(k)}$ are given via

$$
v^{(k)}(t,x)=e^{-bt}\left(v_0(x)+\sum_{i=1}^Na_i\int_0^te^{b\tau}u_i^{(k)}(\tau,x)\di{\tau}\right).
$$
Consequently, we can estimate

\begin{align*}
\left|r^{(1)}-r^{(2)}\right|&\leq2\pi\alpha\sum_{i=1}^N\int_0^t\left(a_i\left|u_i^{(1)}-u_i^{(2)}\right|
+\beta_i\sum_{j=1}^Na_j\int_0^\tau e^{bs}\left|u_j^{(1)}-u_j^{(2)}\right|\di{s}\right)\di{\tau}\\
&\leq C\int_0^t\left(\left|u^{(1)}-u^{(2)}\right|+\int_0^\tau e^{bs}\left|u^{(1)}-u^{(2)}\right|\di{s}\right)\di{\tau}.
\end{align*}
where the constant $C>0$ is independent of the particular choice of $u^{(k)}$ ($k=1,2$).
\end{proof}

\begin{lemma}[Estimates for the cell problem]\label{lem_trans}
Let $\e_2\leq r_1\leq r_2\leq\nicefrac{1}{2}(1-\e_1)$ and let $w^{(i)}_k$, $k,i=1,2$, solve

\begin{alignat*}{2}
-\Delta w_{k}^{(i)}&=0&\quad&\text{in}\ \ S\times Y^{(i)},\\
-\nabla w_{k}^{(i)}\cdot n&=e_k\cdot n&\quad&\text{on}\ \ S\times\Sigma^{(i)},\\
\int_{Y^{(i)}}w_k^{(i)}(y)\di{y}&=0,&\\
y&\mapsto w_k^{(i)}(y)&\quad&\text{is $Y$-periodic}.
\end{alignat*}
Then, the following estimate holds:

$$
\left|\int_{Y^{(1)}}\nabla w_k^{(1)}\cdot e_j\di{y}-\int_{Y^{(2)}}\nabla w_k^{(2)}\cdot e_j\di{y}\right|
\leq C|r^{(1)}-r^{(2)}|,
$$
where the constant $C>0$ might dependent on $e_1$ and $e_2$ but not on the particular choice of $r^{(1)}$ and $r^{(2)}$.
Here, we have set $Y^{(j)}=Y\setminus\overline{B(r^{(j)})}$ and $\Sigma^{(j)}=\partial B(r^{(j)})$.
\end{lemma}
\begin{proof}
We prove this statement in three steps. First, we introduce a coordinate transform that allows us to compare the different solutions and, second, go on proving some important energy estimates.
Finally, we use these energy estimates to proof the desired result.\\[.2cm]

\emph{Step1: Transformation:}
We set $a=(\nicefrac{1}{2},\nicefrac{1}{2})$ and introduce the transformation $\xi\colon\overline{Y}\to\overline{Y}$ given by

$$
\xi(y)=
\begin{cases}
    y,\quad &|y-a|\geq\nicefrac{1}{2},\\ 
    (1-\chi(|y-a|))y+\chi(|y-a|)\left(\nicefrac{r^{(1)}}{r^{(2)}}(y-a)+a\right),\quad&r^{(2)}\leq|y-a|\leq\nicefrac{1}{2},\\
    \nicefrac{r^{(1)}}{r^{(2)}}(y-a)+a,\quad &|y-a|\leq r^{(2)}
\end{cases}
$$
Here, $\chi\colon[r^{(2)},\nicefrac{1}{2}]\to[0,1]$ is a smooth cut-off function with compact support (i.e., $\chi\in C_0^\infty(r^{(2)},\nicefrac{1}{2})$) satisfying $\chi(r^{(2)})=1$, $\chi(\nicefrac{1}{2})=0$, as well as $-\nicefrac{4}{\e_1}\leq\chi'(z)\leq0$.
As a result, $\xi$ is a smooth function as well and satisfies $\xi(Y^{(2)})=Y^{(1)}$ and $n_{\Sigma^{(1)}}(\xi(y))=n_{\Sigma^{(2)}}(y)$ for all $y\in\Sigma^{(2)}$.

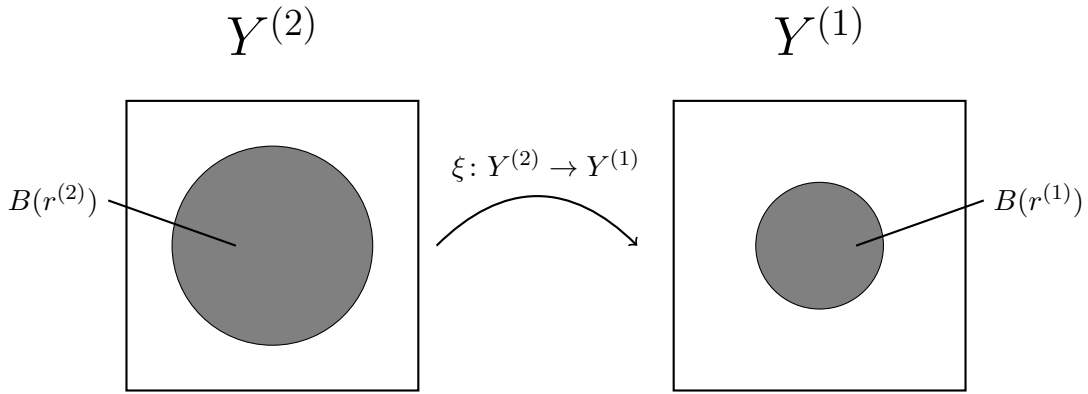
\begin{figure}[htbp]
\centering
\begin{tikzpicture}[scale=1.2]
    \draw[thick] (-1.6,-1.6) rectangle (1.6,1.6);
    \draw[fill=black!50] (0,0) circle (1.1cm);
	\draw (0,2) node[above] {{\huge{$Y^{(2)}$}}};
	\draw[thick] (-0.4,0) -- (-1.8,0.5);
	\draw (-1.8,0.5) node[left] {{$B({r^{(2)}})$}};
    \draw[thick] (-1.6+6,-1.6) rectangle (1.6+6,1.6);
    \draw[fill=black!50] (0+6,0) circle (.7cm);
	\draw (0+6,2) node[above] {{\huge{$Y^{(1)}$}}};
	\draw[thick] (0.4+6,0) -- (1.8+6,0.5);
	\draw (1.8+6,0.5) node[right] {{$B({r^{(1)}})$}};
	\draw[thick,->] (1.8,0)  to [out=45,in=135,looseness=1.2]  (4,0);
	\draw (3,0.6) node[above] {{$\xi\colon Y^{(2)}\to Y^{(1)}$}};
\end{tikzpicture}
\caption{Sketch of the transformation connecting reference cells for different radii $r^{(1)}$ and $r^{(2)}$.} 
    \label{fig:transform}
\end{figure}

Calculating the Jacobi matrix for $\xi$, we see that $D\xi=\mathds{I}_2$ for $|y-a|\geq\nicefrac{1}{2}$ and $D\xi=\left(\nicefrac{r^{(1)}}{r^{(2)}}\right)^2\mathds{I}_2$ for $|y-a|\leq r^{(2)}$.
For the transition part, i.e., $r^{(2)}\leq|y-a|\leq\nicefrac{1}{2}$, we calculate

\begin{align*}
\partial_{y_i}\xi_j(y)&=
    \partial_{y_i}\left[y\mapsto(1-\chi(|y-a|))y_j+\chi(|y-a|)\left(\nicefrac{r^{(1)}}{r^{(2)}}(y_j-\nicefrac{1}{2})+\nicefrac{1}{2}\right)\right]\\
    &=\delta_{ij}\bigg(1+(\nicefrac{r^{(1)}}{r^{(2)}}-1)\chi(|y-a|)\bigg)
   +\left(\nicefrac{r^{(1)}}{r^{(2)}}(y_j-\nicefrac{1}{2})+\nicefrac{1}{2}-y_j\right)\frac{y_i-\nicefrac{1}{2}}{|y-a|}\chi'(|y-a|)
\end{align*}
As a consequence, we find that the Jacobian is given by the symmetric matrix

\begin{align}
    \label{deter}
	D\xi(y)=a(|y-a|)\begin{pmatrix}1&0\\0&1\end{pmatrix}
	+b(|y-a|)\begin{pmatrix}
		(y_1-\nicefrac{1}{2})^2& 
		(y_1-\nicefrac{1}{2})(y_2-\nicefrac{1}{2})\\
		(y_1-\nicefrac{1}{2})(y_2-\nicefrac{1}{2})&
		(y_2-\nicefrac{1}{2})^2
	\end{pmatrix}
\end{align}
where (setting $\overline{r}=r^{(2)}-r^{(1)}\geq0$)

\begin{align*}
a(z)=\left(1-\frac{\overline{r}}{r^{(2)}}\chi(z)\right),\quad b(z)=-\frac{\chi'(z)}{z}\frac{\overline{r}}{r^{(2)}}.
\end{align*}
We can calculate the determinant as
$$
\det D\xi(y)=a(|y-a|)\big(a(|y-a|)+b(|y-a|)\left(y_1^2-y_1+y_2^2-y_2+1\right)\big).
$$
Since $a(|y-a|)>0$, $b(|y-a|)\geq0$ and $y_1^2-y_1+y_2^2-y_2+1>0$ for all $y=(y_1,y_2)\in Y$, we find that
$$
\det D\xi(y)\geq\inf_{r^{(2)}\leq|y-a|\leq\nicefrac{1}{2}}a^2(|y-a|)=\left(\frac{r^{(1)}}{r^{(2)}}\right)^2.
$$
This shows that 

$$
4\e_2^2\leq\left(\frac{\e_2}{\nicefrac{1}{2}(1-\e_1)}\right)^2\leq\det D\xi(y)\leq1
$$
which implies invertibility of $D\xi$.\\[.3cm]

\emph{Step 2: Energy estimates}.
In the following, we set $F(y)=D\xi(y)$ and $J(y)=|\det F(y)|$.
We start with the the weak forms

$$
\int_{Y^{(i)}}\nabla w^{(i)}_{k}\cdot\nabla \eta^{(i)}\di{z}=\int_{\Sigma^{(i)}}e_k\cdot n_{\Sigma^{(i)}}\eta^{(i)}\di{\sigma} \quad \left(\eta^{(i)} \in H^1_\#(Y^{(i)}),\ i=1,2\right).
$$
We take the difference of these two weak forms:

$$
\int_{Y^{(1)}}\nabla w^{(1)}_{k}\cdot\nabla \eta^{(1)}\di{y}-\int_{Y^{(2)}}\nabla w^{(2)}_{k}\cdot\nabla \eta^{(2)}\di{y}=e_k\cdot\left[\int_{\Sigma^{(1)}}n_{\Sigma^{(1)}}\eta^{(1)}\di{\sigma}-\int_{\Sigma^{(2)}}n_{\Sigma^{(2)}}\eta^{(2)}\di{\sigma}\right].
$$
and transform the surface integral on the right-hand side in order to arrive at
$$
\int_{\Sigma^{(1)}}n_{\Sigma^{(1)}}\eta^{(i)}\di{\sigma}-\int_{\Sigma^{(2)}}n_{\Sigma^{(2)}}\eta^{(2)}\di{\sigma}=\int_{\Sigma^{(2)}}n_{\Sigma^{(1)}}(\xi(y))\eta^{(1)}(\xi(y))|\det D\xi(y)|\di{\sigma}-\int_{\Sigma^{(2)}}n_{\Sigma^{(2)}}\eta^{(2)}\di{\sigma}.
$$
%
By construction, we have $n_{\Sigma^{(1)}}(\xi(y))=n_{\Sigma^{(2)}}(y)$ for all $y\in\Sigma^{(2)}$ leading to

\begin{align*}
\int_{\Sigma^{(1)}}n_{\Sigma^{(1)}}\eta^{(1)}\di{\sigma}-\int_{\Sigma^{(2)}}n_{\Sigma^{(2)}}\eta^{(2)}\di{\sigma}&=\int_{\Sigma^{(2)}}\bigg(\eta^{(1)}(\xi(y))\det D\xi(y)-\eta^{(2)}(y)\bigg)n_{\Sigma^{(2)}}\di{\sigma}\\
&=\int_{\Sigma^{(2)}}\bigg(\eta^{(1)}(\xi(y))-\eta^{(2)}(y)\bigg)\det D\xi(y)n_{\Sigma^{(2)}}\di{\sigma}\\
&\qquad+\int_{\Sigma^{(2)}}\bigg(\det D\xi(y)-1\bigg)\eta^{(2)}(y)n_{\Sigma^{(2)}}\di{\sigma}
\end{align*}
For the volume integral on the l.h.s., we get (note that the Jacobian is symmetric)

\begin{multline*}
\int_{Y^{(1)}}\nabla w^{(1)}_{k}\cdot\nabla \eta^{(1)}\di{y}-\int_{Y^{(2)}}\nabla w^{(2)}_{k}\cdot\nabla \eta^{(2)}\di{y}\\
=\int_{Y^{(2)}}\det D\xi(D\xi)^{-2}\nabla w^{(1)}_{k}(\xi)\cdot\nabla \eta^{(1)}(\xi)-\nabla w^{(2)}_{k}\cdot\nabla \eta^{(2)}\di{y}
\end{multline*}
and, as a consequence,
\begin{multline*}
\int_{Y^{(2)}}\det D\xi(D\xi)^{-2}\nabla w^{(1)}_{k}(\xi)\cdot\nabla \eta^{(1)}(\xi)-\nabla w^{(2)}_{k}\cdot\nabla \eta^{(2)}\di{y}\\
=\int_{\Sigma^{(2)}}\bigg(\eta^{(1)}(\xi(y))-\eta^{(2)}(y)\bigg)\det D\xi(y)n_{\Sigma^{(2)}}\di{\sigma}
+\int_{\Sigma^{(2)}}\bigg(|\det D\xi(y)|-1\bigg)\eta^{(2)}(y)n_{\Sigma^{(2)}}\di{\sigma}.
\end{multline*}
Now, choosing $\widetilde{\eta}^1=\eta^{(2)}=\widetilde{w}_k^{(1)}-w_k^{(2)}=:\overline{w}_k$, this leads to

\begin{multline*}
\|\nabla \overline{w}_k\|^2_{L^2(Y^{(2)})}\leq
\int_{Y^{(2)}}\left|\det D\xi(D\xi)^{-2}-\mathds{I}_2\right|\left|\nabla \widetilde{w}^{(1)}_{k}\right|\cdot\left|\nabla \overline{w}_k\right|\di{y}\\
+\int_{\Sigma^{(2)}}\bigg|\det D\xi(y)-1\bigg|\left|\overline{w}_k\right|\di{\sigma}.
\end{multline*}
For $y\in\Sigma^{(2)}$, i.e., $|y-a|=r^{(2)}$, we have

$$
1-\det D\xi(y)=1-\left(\frac{r^{(1)}}{r^{(2)}}\right)^2=\frac{(r^{(2)})^2-(r^{(1)})^2}{(r^{(2)})^2}\leq\frac{\overline{r}}{r^{(2)}}.
$$
Now, for $y\in Y_2$ with $|y-a|\geq\nicefrac{1}{2}$, we have $\det D\xi=1$ and $D\xi=\mathds{I}_2$ and, in the case that $r^{(2)}\leq|y-a|\leq \nicefrac{1}{2}$,

$$
\left|\det D\xi(D\xi)^{-2}-\mathds{I}_2\right|\leq\frac{\left|\det D\xi-1\right|}{|D\xi|^{2}}+\frac{\left|(D\xi)^{-1}-\mathds{I}_2\right|}{|D\xi|}+\left|(D\xi)^{-1}-\mathds{I}_2\right|
$$
Since $|D\xi|^2\geq \det D\xi\geq 4\e_2^2$ and $1-\det D\xi(y)\leq \nicefrac{\overline{r}}{r^{(2)}}$:

$$
\left|\det D\xi(D\xi)^{-2}-\mathds{I}_2\right|\leq\frac{\overline{r}}{4r^{(2)}\e_2^2}+\left(1+\frac{1}{2\e_2}\right)\left|(D\xi)^{-1}-\mathds{I}_2\right|
$$
Finally, via 
$$
\left|(D\xi)^{-1}-\mathds{I}_2\right|\leq \left|(D\xi)^{-1}\right|\left|\mathds{I}_2-D\xi\right|\leq 2\e_2\left|\mathds{I}_2-D\xi\right|
$$
we arrive at (looking at \cref{deter})

$$
\left|\det D\xi(D\xi)^{-2}-\mathds{I}_2\right|
\leq\frac{\overline{r}}{r^{(2)}}\left(\frac{1}{4\e_2^2}+2\e_2+1+\frac{1}{\e_1r^{(2)}}\right)
$$
Therefore we find that

$$
\|\nabla \overline{w}_k\|^2_{L^2(Y^{(2)})}\leq
C(\e_1,\e_2)\overline{r}\left(\int_{Y^{(2)}}\left|\nabla \widetilde{w}^{(1)}_{k}\right|\cdot\left|\nabla \overline{w}_k\right|\di{y}
+\int_{\Sigma^{(2)}}\left|\overline{w}_k\right|\di{\sigma}\right).
$$
Applying Poincar\'e's inequality (possible due to the zero average condition) and the trace theorem leads to the energy estimate

\begin{align}
\label{tr_energy}
\|\overline{w}_k\|_{H^1(Y^{(2)})}\leq \tilde{C}(\e_1,\e_2)\overline{r},
\end{align}
where the constant $\tilde{C}(\e_1,\e_2)>0$ is independent of $r^{(1)}$ and $r^{(2)}$.\\[.3cm]

\emph{Step 3: Proving the result}. Using \cref{tr_energy}, we go on by estimating  the following key expression:

\begin{align*}
\left|\int_{Y^{(1)}}\nabla w_k^{(1)}\cdot e_j\di{y}-\int_{Y^{(2)}}\nabla w_k^{(2)}\cdot e_j\di{y}\right|
&\leq \left|\int_{Y^{(1)}}\nabla w_k^{(1)}\di{y}-\int_{Y^{(2)}}\nabla w_k^{(2)}\di{y}\right|\\
&=\left|\int_{Y^{(2)}}\det D\xi (D\xi)^{-1}\nabla \widetilde{w}_k^{(1)}-\nabla w_k^{(2)}\di{y}\right|\\
&\leq \widehat{C}(\e_1,\e_2)\overline{r},
\end{align*}

\end{proof}

\subsection{A fixed-point argument}
Now, let $\e_1,\e_2, M^*, s^*>0$ and initial conditions $r_0, v_0$ be chosen such that $2\e_2\leq 2r_u(t,x)\leq1-\e_1$ for all $(t,x)\in(0,s^*)\times\Omega$ and all $u\in T_{s^*,M^*}$ (this is possible due to \Cref{lemma_boundsr,lemma_diffus}).
Also, let $0\leq u_{i0}(x)\leq\nicefrac{M^*}{2}$.
These choices imply $F(u)=(F_1(u),...,F_N(u))\in L^\infty((0,s^*)\times\Omega)^N$ for all $u\in T_{s^*,M^*}$ (see \cref{eq:rhs}).
In the following, let $s\in (0,s^*)$ and $M\in(0,M^*)$.

We will now look at the linearized problem: For some $\tilde{u}\in T_{s,M}$, we try to find a function $u\in W((0,s);H^1(\Omega))$ solving

\begin{subequations}
\begin{alignat}{2}
	\partial_tu_i-\dive\left(\widehat{D_i}(\tilde{u})\nabla u_i\right)&=F_i(\tilde{u})&\quad&\text{in}\ \ S\times\Omega,\label{lina}\\
	-\widehat{D_i}(\tilde{u})\nabla u_i\cdot n&=0&\quad&\text{on}\ \ S\times\partial\Omega,\label{linb}\\
	u_i(0)&=u_{i0}&\quad&\text{in}\ \ \Omega.\label{linc}
\end{alignat}
\end{subequations}

\begin{lemma}[Existence result for linearized problem]
\label{existence_linear}
For each $\tilde{u}\in T_{s,M}$, there is a unique $u\in W((0,s);H^1(\Omega))$ solving the problem given by \cref{lina,linb,linc}.
Moreover, the following a priori estimates are satisfied
\begin{multline*}
\|\partial_tu\|^2_{L^2(S;H^1(\Omega)^*)}+\|u\|_{L^\infty((0,s);L^2(\Omega))}^2+\|\nabla u\|_{L^2((0,s)\times\Omega)}^2\\
\leq C\left(\|u_0\|^2_{L^2(\Omega)}+\sum_{i=1}^N\|F_i(\tilde{u})\|_{L^\infty((0,s)\times\Omega)}\right)
\end{multline*}
where the constant $C>0$ does not depend on $\tilde{u}$, $s$, and $M$.
Please note that the above estimate implies boundedness in $W((0,s);H^1(\Omega))$ as well.
\end{lemma}
\begin{proof}
Since $\tilde{u}\in T_{s,M}$, we have $F_i(\tilde{u})\in L^\infty((0,s)\times\Omega)$ ($i=1,...,N$). 
Also, the diffusivity matrix $\widehat{D_i}(\tilde{u})$ is uniformly positive definite (i.e., there is $c_i>0$ such that $\widehat{D_i}(\tilde{u})(t,x)\xi\cdot\xi\geq c_i|\xi|^2$ for all $(t,x)\in (0,s)\times\Omega$ and all $\xi\in\R^3$).
Finally, as the $D_i$ are also bounded, the existence of a unique solution follows by standard theory of parabolic PDE.

To search for the needed {\em a priori} estimates, we test the weak form with $u_i$. Hence, we are led to

$$
\|u_i(t)\|^2_{L^2(\Omega)}+2c_i\int_0^t\|\nabla u_i\|^2_{L^2(\Omega)}\di{\tau}\leq\|u_{i0}\|^2_{L^2(\Omega)}+2\int_0^t\int_\Omega|F_i(\tilde{u})u_i|\di{x}\di{\tau}\quad (t\in(0,s)).
$$
From here, summing over $i=1,...,N$ and applying Grönwall's inequality leads to the desired estimate for $u$ and $\nabla u$.
Similarly, taking a test function $\varphi\in L^2((0,s);H^1(\Omega))$ such that $\|\varphi\|\leq1$, we find that

$$
\langle\partial_tu_i,\varphi\rangle_{L^2((0,s);H^1(\Omega)^*)}\leq\int_\Omega |F_i(\tilde{u})\varphi|\di{x}+\int_\Omega|\widehat{D_i}(\tilde{u})\nabla u_i\nabla\varphi|\di{x}
$$
thus completing the estimate.
\end{proof}

With the solvability of the linarized problem established, we want to investigate under what circumstances we can ensure that $u\in T_{s,M}$ as well; as this would then naturally lead to a potential fixed-point scheme.
As a first point, any $\widetilde{u}\in T_{s,M}$ leads to a solution $u\in W((0,s);H^1(\Omega))$ which again leads to the corresponding solution operator

$$
\mathcal{L}\colon T_{s,M}\to W((0,s);H^1(\Omega))^N.
$$
We now need to show, that $s\in(0,s^*)$ and $M\in(0,M^*)$ can be chosen such that $\mathcal{L}[T_{s,M}]\subset T_{s,M}$.
With the following lemma, we first establish $\mathcal{L}[T_{s,M}]\subset L^\infty((0,s)\times\Omega)^N$.

\begin{lemma}[Boundedness]
\label{lemma_bounded}
For every $\tilde{u}\in T_{s,M}$, the solution of the linearized equation is bounded by
$$
-t\esssup(F_i(\tilde{u}))_-\leq u_i\leq \esssup u_{i0}+t\esssup F_i(\tilde{u}).
$$
In particular, we have $u\in L^\infty((0,s)\times\Omega)^N$.
\end{lemma}

\begin{proof}
By the linearity of the problem, we can decompose the solution $u_i=\pi_i+\omega_i$, where 

\begin{minipage}[t]{0.5\textwidth}
\begin{alignat*}{2}
	\partial_t\pi_i-\dive\left(\widehat{D_i}(\tilde{u})\nabla \pi_i\right)&=0&\quad&\text{in}\ \ S\times\Omega,\\
	-\widehat{D_i}(\tilde{u})\nabla \pi_i\cdot n&=0&\quad&\text{on}\ \ S\times\partial\Omega,\\
	\pi_i(0)&=u_{i0}&\quad&\text{in}\ \ \Omega,
\end{alignat*}
\end{minipage}\vspace{.08cm}
\begin{minipage}[t]{0.5\textwidth}
\begin{alignat*}{2}
	\partial_t\omega_i-\dive\left(\widehat{D_i}(\tilde{u})\nabla \omega_i\right)&=F_i(\tilde{u})&\quad&\text{in}\ \ S\times\Omega,\\
	-\widehat{D_i}(\tilde{u})\nabla \omega_i\cdot n&=0&\quad&\text{on}\ \ S\times\partial\Omega,\\
	\omega_i(0)&=0&\quad&\text{in}\ \ \Omega.
\end{alignat*}
\end{minipage}
Estimating the $\pi_i$-problems via $(\pi_i-L_i)_+$ for $L_i=\esssup u_{i0}$, we find that $\pi_i\leq L_i$.
Using Duhamel's principle, we get $\omega_i(t,x)=\int_0^t h_i(\tau,t,x)\di{\tau}$ where the $\tau$-parametrized function $h_i$ solves

\begin{alignat*}{2}
	\partial_th_i-\dive\left(\widehat{D_i}(\tilde{u})\nabla h_i\right)&=0&\quad&\text{in}\ \ S\times\Omega,\\
	-\widehat{D_i}(\tilde{u})\nabla h_i\cdot n&=0&\quad&\text{on}\ \ S\times\partial\Omega,\\
	h_i(0)&=F_i(\tilde{u}(\tau,\cdot))&\quad&\text{in}\ \ \Omega.
\end{alignat*}
This implies $h_i\leq\esssup(F_i(\tilde{u}))_+$ and, as a consequence $\omega_i\leq t\esssup(F_i(\tilde{u}))_+$.
Finally, we have

$$
u_i\leq \esssup u_{i0}+t\esssup(F_i(\tilde{u}))_+.
$$
Now, since $u_{i0}\geq0$, we find that $\pi_i\geq0$ as well.
Testing with $(h_i+\esssup(F_i(\tilde{u}))_-)_-$, we arrive at $h_i\geq-\esssup(F_i(\tilde{u}))_-$ and, as a consequence $\omega_i\geq-t\esssup(F_i(\tilde{u}))_-$.
This shows
$$
u_i\geq-t\esssup(F_i(\tilde{u}))_-.
$$
In particular, we find that $u_i\in L^\infty((0,s)\times\Omega)$ with 
$$\|u_i(t)\|_{L^\infty(\Omega)}\leq \|u_{i0}\|_{L^\infty(\Omega)}+t\|F_i(\tilde{u})(t)\|_{L^\infty(\Omega)}.$$
\end{proof}
Now, in order to get concrete bounds for the solution $u=(u_1,...,u_N)$, we have to take a closer look at the right-hand sides:
For the $F_i(\tilde{u})$, we have the estimates (given our assumptions on $r_0$, $s^*$, and $M^*$ and using \cref{est_F1,est_F2,est_F3}):
\begin{align*}
F_i(\tilde{u})&\leq M\left(M\gamma\left(N+\frac{k+1}{2}\right)+15\left(a_i+\frac{a}{b}\beta_i(e^{bt}-1)\right)\right)+15\beta_iv_0(x),\\
F_i(\tilde{u})&\geq-M\left(M\gamma\left(N+\frac{k+1}{2}\right)+15\left(a_i+\frac{a}{b}\beta_i(e^{bt}-1)\right)\right)
\end{align*}
or, more compactly,

\begin{align}\label{estimate_rhs}
\|F_i(\tilde{u})(t)\|_{L^\infty(\Omega)}\leq15\beta_i\|v_0\|_{L^\infty(\Omega)}+M\left(M\gamma\left(N+\frac{k+1}{2}\right)+15\left(a_i+\frac{a}{b}\beta_i(e^{bt}-1)\right)\right).
\end{align}
With this estimate at hand, we are now able to establish that $\mathcal{L}$ is a self-mapping for a suitable choice of $(s,M)$.

\begin{lemma}[Fixed-point operator]
\label{lemma_fixed}
For any $M\in(0,M^*)$ there is $s\in(0,s^*)$ such that for every $\tilde{u}\in T_{s,M}$ the solution $u$ of the linearized problem also satisfies $u=\mathcal{L}(\tilde{u})\in T_{s,M}$. 
\end{lemma}
\begin{proof}
For any given $M\in(0,M^*)$, we find that
$$
\lim_{t\to0}t\|F_i(\tilde{u})\|_\infty\to0\quad(i=1,...,N).
$$
uniformly for $\tilde{u}\in T_{s,M}$ (see inequality \ref{estimate_rhs}).
As a consequence, it is possible to find $s\in(0,s^*)$ such that $s\|F_i(\tilde{u})\|_\infty\leq\nicefrac{M}{2}$ for all $i=1,...,N$ and for all $\tilde{u}\in T_{s,M}$.
This implies $u\in T_{s,M}$ via \Cref{lemma_bounded}.
\end{proof}

Please note that $T_{s,M}$ is a closed subset of $L^2((0,s)\times\Omega)$. 
In the following lemma we investigate continuity of the fixed point operator

\begin{lemma}[Continuity]
\label{lemma_cont}
The operator 
$$
\mathcal{L}\colon T_{s,M}\to L^2((0,s)\times\Omega)
$$
is continuous with respect to the $L^2$-norm.
\end{lemma}
\begin{proof}
Now let $\tilde{u},\tilde{u}^{(k)}\in T_{s,M}$ such that $\tilde{u}^{(k)}\to \tilde{u}$ in $L^2((0,s)\times\Omega)$ for $k\to\infty$.
In addition, let $u=\mathcal{L}(\tilde{u})$ and  $u^{(k)}=\mathcal{L}(\tilde{u}^{(k)})$ ($k\in\N)$ be the corresponding unique solutions to the linearized problem (see \Cref{existence_linear}).

Now, the sequence $u^{(k)}$ is bounded in $W((0,s);H^1(\Omega))$ since $0\leq\tilde{u}^{(k)}\leq M$ and the a priori estimates given by \Cref{existence_linear}.
Since $W((0,s);H^1(\Omega))$ is a reflexive Banach space and since it is compactly embedded in $L^2((0,s)\times\Omega)$ (\emph{Lions-Aubin lemma}), there is a subsequence (for ease of notation, still denoted by $u^{(k)}$) and a limit function $u^*$ such that $u^{(k)}$ converges to $u^*$ strongly and weakly in $L^2((0,s)\times\Omega)$ and $W((0,s);H^1(\Omega))$, respectively.
Without loss of generality, we also have $u^{(k)}\to u$ pointwise almost everywhere over $(0,s)\times\Omega$ (possibly by choosing a further subsequence).
In the following, we show continuity by establishing that $u^*=u$.\footnote{Due to this resulting statement: \emph{Every subsequence has a further subsequence converging to $u$.}}

The components of $u^{(k)}$ satisfy (for all $\varphi\in H^1(\Omega)$ and $t\in(0,s)$)
\begin{align*}
	\langle\partial_tu_i^{(k)},\varphi\rangle_{H^1(\Omega)^*}+\int_\Omega\widehat{D_i}(\tilde{u}^{(k)})\nabla u^{(k)}\cdot\nabla\varphi\di{x}&=\int_\Omega F_i(\tilde{u}^{(k)})\varphi\di{x}.
\end{align*}
Now, since $\tilde{u}^{(k)}\to \tilde{u}$ in $L^2((0,s)\times\Omega))$, it holds

$$
\int_\Omega F_i(\tilde{u}^{(k)})\varphi\di{x}\to\int_\Omega F_i(\tilde{u})\varphi\di{x} \quad (\varphi\in H^1(\Omega),\, i=1,..,N).
$$
For the diffusion term, we take a look at

\begin{multline*}
    \int_\Omega\left(\widehat{D_i}(\tilde{u})\nabla u^*-\widehat{D_i}(\tilde{u}^{(k)})\nabla u^{(k)}\right)\cdot\nabla\varphi\di{x}\\
    =\int_\Omega\widehat{D_i}(\tilde{u})\nabla\left(u^*-u^{(k)}\right)\cdot\nabla\varphi\di{x}
    +\int_\Omega\left(\widehat{D_i}(\tilde{u})-\widehat{D_i}(\tilde{u}^{(k)})\right)\nabla u^{(k)}\cdot\nabla\varphi\di{x}.
\end{multline*}
Here, the first term on the right hand side goes to zero due to the weak convergence of $u^{(k)}$ to $u^*$ in $W((0,s);H^1(\Omega))$.
Looking at the second term, we recall

\begin{multline*}
\left(\widehat{D_i}(\tilde{u})-\widehat{D_i}(\tilde{u}^{(k)})\right)_{lm}\\
=d_i\left(\phi(r^{(0)})\int_{Y^{(0)}}(\nabla w^{(0)}_{l}+e_{l})\cdot e_m\di{z}-\phi(r^{(k)})\int_{Y^{(k)}}(\nabla w^{(k)}_{l}+e_{l})\cdot e_m\di{z}\right),
\end{multline*}
which can be estimated using \Cref{lem_rad,lem_trans} 

$$
\left|\widehat{D_i}(\tilde{u})-\widehat{D_i}(\tilde{u}^{(k)})\right| 
\leq C\int_0^t\left(\left|\tilde{u}-\tilde{u}^{(k)}\right|+\int_0^\tau e^{bs}\left|\tilde{u}-\tilde{u}^{(k)}\right|\di{s}\right)\di{\tau}.
$$
Here, we have used for the porosity that

$$
\left|\phi(r^{(0)})-\phi(r^{(k)})\right|\leq\frac{\pi^2}{|\Omega|}\left|r^{(0)}-r^{(k)}\right|.
$$
Now, since $\tilde{u}^{(k)}\to\tilde{u}$ almost everywhere over $(0,s)\times\Omega$, dominated convergence leads to 

$$
\int_\Omega\left(\widehat{D_i}(\tilde{u})-\widehat{D_i}(\tilde{u}^{(k)})\right)\nabla u^{(k)}\cdot\nabla\varphi\di{x}\to0
$$
As a consequence, $u^*=u$.
\end{proof}


\begin{theorem}[Existence]\label{existence}
The operator 
$$
\mathcal{L}\colon T_{s,M}\to L^2((0,s)\times\Omega)
$$
has at least one fixed-point $u^*\in W((0,s);H^1(\Omega))$.
\end{theorem}
\begin{proof}
$T_{s,M}$ is a non-empty, closed, and convex subset of $L^2((0,s)\times\Omega)$ and $\mathcal{L}$ is continuous with respect to the $L^2((0,s)\times\Omega)$ norm (\Cref{lemma_cont}).
Moreover, we have $\mathcal{L}[T_{s,M}]\subset T_{s,M}$ via \Cref{lemma_fixed}. 
Finally, since $\mathcal{L}[T_{s,M}]\subset W((0,s);H^1(\Omega))$ which is compactly embedded in $L^2((0,s)\times\Omega)$ by virtue of Lions-Aubin's lemma, we can employ Schauder's fixed point thorem to conclude the existence of at least one fixed-point $u^*\in W((0,s);H^1(\Omega))\cap T_{s,M}$.
\end{proof}

\begin{remark}
Relying for instance on techniques from \cite{degenerate}, we expect the weak solution given by \Cref{existence} to be of higher regularity provided that data (boundary of $\Omega$, initial conditions) are sufficiently smooth.
This could change, however, if we were to allow actual clogging of the porous medium.
\end{remark}



\section{Numerical simulation of the two-scale quasilinear  problem}\label{numerics}

\subsection{Setup of the model equations and target geometry}

The aim is to solve numerically the two-dimensional macroscopic model problem
for the species concentration $u_i$ ($i\in\{1,\dots,N\}$) and  $v$. To focus the attention on physically relevant choices of parameters, we use the setup described in \cite{johnson1995dynamics}; see also \cite{Krehel,MC20} for more details. Essentially, we look at a theoretical model describing the dynamics of colloid
deposition on collector surfaces,  when  both inter-particle, and particle-surface electrostatic interactions   are assumed to be negligible. The numerical range of the used parameters fit to the situations that can relate to the immobilization of bio-colloids in soils. 

The simulation output we are looking after includes approximated space and time concentration profiles of colloidal populations, spatial distribution of microstructures for given time slices, and  estimated amount of deposited colloidal mass. This information helps us detect in {\em a posteriori} way the locations in $\Omega$ where deposition-induced clogging is likely to happen.


We have
\bse\label{mod1}
\be
\hspace{-1cm}{ \partial_t} u_i(x,t) = D_{ijk}(x,t) \Delta_x u_i(x, t)
+ R_i(u)\nonumber
 -\frac{{L}(x,t)}{A(x,t)} \left(a_i u_i(x,t)-\beta_i v(x,t)\right),\label{e1}
\ee
describing the diffusion of $u_i$ in the macroscopic domain $\Omega$.

The effective diffusion tensor has the form  
 $$D_{ijk}(x,t)= d_i \phi(x,t)\tau_{jk}(x,t),$$ where the entries
  $$\tau_{jk}(x,t)=\int_{Y(x,t)} \left( \delta_{j,k}+\nabla_{y_j}w_k(z,t) \right) dz,$$ 
    for all $i=1,\ldots,N$, $j,k=1,2$.

 In addition, the length $L$ and area $A$ functions related to the motion of the boundary (for $r<1/2$) are:
   \be \label{mod1ha}
{L}(x,t)=\int_{\Gamma(x,t)}ds=2\pi r(x,t),\quad
A(x,t)=\int_{ Y_0(x,t)}dy=1-\pi r^2(x,t),\quad \mbox{(in 2D)}
\ee
\be \label{mod1hb}
R_i(u)=\frac12\sum_{i+j=k}\alpha_{i,j}\beta_{i,j} u_i u_j -
 u_k\sum_{i=1}^\infty \alpha_{k,i}\beta_{k,i} u_i.
\ee

Moreover, the cell functions $w:=(w_1(x,y,t),w_2(x,y,t))$, assumed to have constant mean, satisfy
 \be \label{mod1hc}
-\Delta_y w_i=0,\quad i=1,2 \quad \mbox{in}\quad  Y_0(x,t),\\
\ee
\be \label{mod1hd}
 - n_0(x,t)\cdot\nabla_y w_i=0, \quad \mbox{on}\quad \partial Y,\quad
 - n_0(x,t)\cdot\nabla_y w_i= n_i(x,t), \quad \mbox{on}\quad \partial B(r).
\ee
with  $\Gamma_e:=\partial Y$ being the boundary of the cell $n_0(x,t)=(n_1(x,t),n_2(x,t))$
 is the corresponding normal vector. 
 
Equation (\ref{e1})
needs to be complemented with corresponding initial and boundary conditions. In the sequel of this section, we focus the discussion on the case of a two dimensional macroscopic domain, i.e. $x=(x_1,x_2)\in[0,1]\times [0,1]$. 

We set Robin conditions at the one side of the square
 \be \label{mod1b}
  \frac{\partial u_i}{\partial n}(x_1,0,t)+ b_r u_i (x_1,0,t)=\left\{\begin{array}{cc}
 u_i^b(x_1)>0 & t\in [0,t_0],\\
 0         & t>t_0,
 \end{array}\right., \quad x_1\in [0,1],
 \ee
while we impose Neumann boundary conditions for the rest of the boundary 
\be \label{mod1b2}
 \frac{\partial u_i}{\partial n}(x_1,x_2, t)=0,
\ee
for $(x_1,x_2)$ such that $0\leq x_2\leq 1$ with   
$x_1=0,1$ or $0\leq x_1\leq 1$ with $x_2=0$
 and with initial conditions
\be \label{mod1c}
  u_i(x,0)=u_i^a(x)\geq 0.
  \ee
Moreover, we have
\be\label{mod1av}
{ \partial_t} v (x, t)=  \sum_{i=1}^N \alpha_i u_i(x,t)-\beta v(x,t),
\ee
 with some initial condition
  \be \label{mod1bv}
  v(x,0)= v_a(x)\geq 0,
  \ee 
 and  
\be  
 r(x,t)\, { \partial_t} r (x,t)=
 \alpha\left( \sum_{i=1}^N a_i u_i(x,t)-\beta  v(x,t)\right)
 {L}(x,t), \label{mod1R}
\ee
together with some initial distribution
\be \label{mod2R}
  r(x,0)=r_a(x)>0, 
\ee
 \ese
 for $ x\in [0,1]\times [0,1]$.
  We discuss in Section  \ref{simulation} additional choices of suitable initial and boundary conditions. 

\subsection{Discretization schemes}\label{simulation}
To treat   problem \eqref{mod1} numerically, we need  
to  obtain firstly a   numerical approximation for the cell problems  \eqref{mod1hc} and determine
 the shape of the corresponding cell functions  $w_1,w_2$ posed in $Y_0(x,t)$. 
 
 More specifically, we proceed for the various values of  $r$,  for $r_a\leq r(x,t) \leq 1/2$. We  take a partition of width $\delta r$, $r_a=r_0, r_1=r_0+\delta r,\ldots, r_{M_1}=1/2$.

 Then since $ Y_0$ is determined as the area contained inside the square cell and outside  the circle of radius $r$,  we obtain a sequence of solutions for the cell problem \eqref{mod1hc} for each 
$ {Y_0}_i$ corresponding to the radius $r_i$ of the partition. 

 We  use a finite element scheme to solve these cell problems. To be precise, we use the MATLAB finite element  package ''\texttt{Distmesh}" (see details in \cite{Persson}) to triangulate the domain ${ Y_0}_i= Y_0(r_i)$. Furthermore, a solver has been implemented to handle this specific problem (equations \eqref{mod1hc}); it works in a similar fashion as applied in \cite{MC20}. 

In Figure \ref{Figw1}, we illustrate the numerical solution for this problem  for a particular choice of $r_i$.
Specifically, we choose to look at $r_i=.25$.
\begin{figure}[htb]
\vspace*{-2cm}
\begin{center}
\includegraphics[bb= 330 230 250 600, scale=.9]{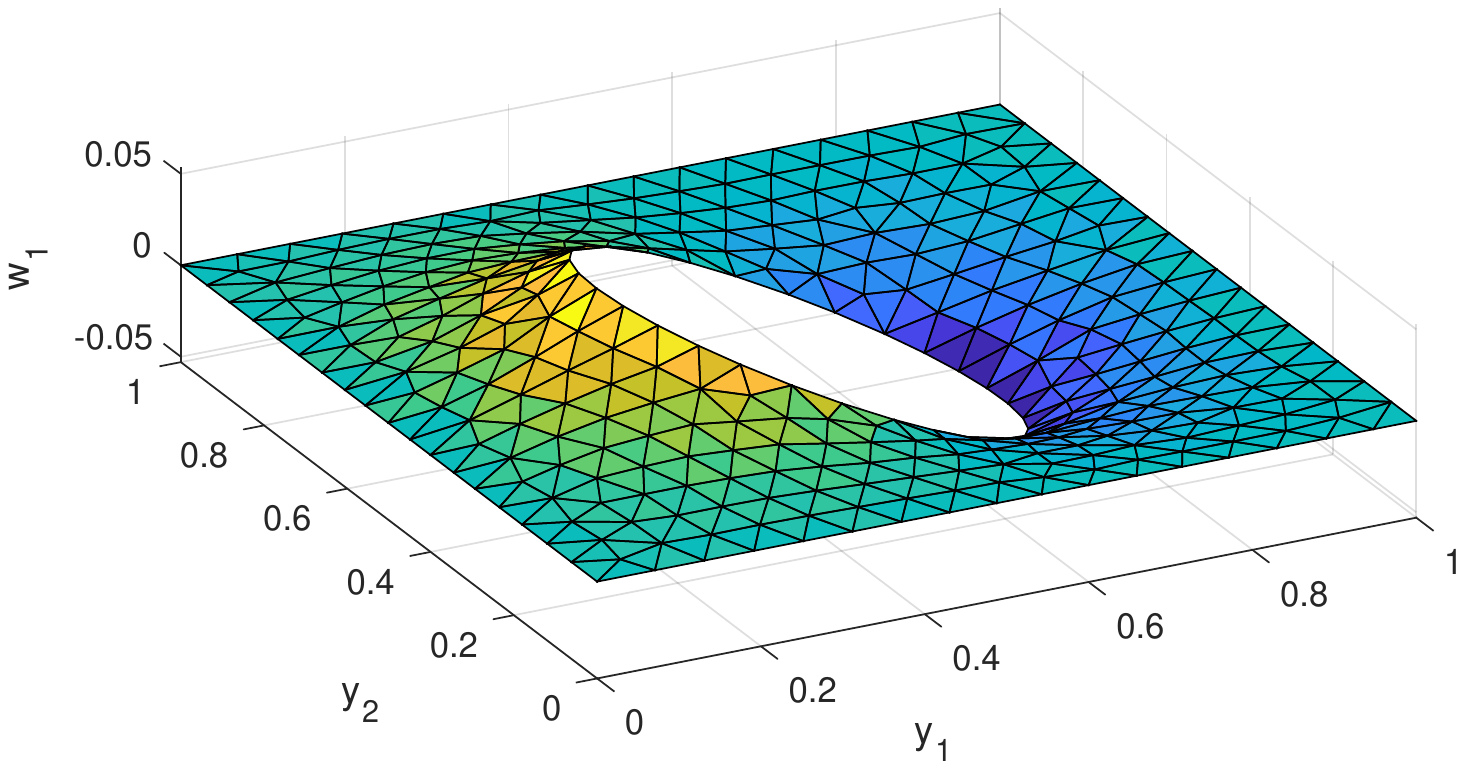}
 \vspace*{-2cm}
\end{center}
\caption{\it Numerical solution of the cell problem \eqref{mod1hc} and specifically for $w_1$ with $r_i=.25$.}\label{Figw1}
\end{figure}

 Having available the numerical evaluation of the cell functions $w$ as approximate solutions to  the cell problems \eqref{mod1hc} and \eqref{mod1hd}, the entries of the diffusion tensor  
 $D_{ijk}=\int_{ Y_0(x,t)} d_i \left( \delta_{j,k}+\nabla_{y_j}w_k \right)$, $i=1,\ldots,N$, $j,k=1,2$ 
 can be calculated directly and for each $(x,t)$ and consequently for the corresponding value for $r(x,t)$ and thus for  $ Y_0(x,t)$. Then the corresponding value of $D_{ijk}(x,t)$ is approximated via linear interpolation.


Next, we solve the system of equations \eqref{e1}-\eqref{mod2R}. 
We use a finite difference scheme to solve the  two-dimensional version of the field equation \eqref{e1},
 together with its boundary and initial conditions. 
 More specifically we consider a square domain $\Omega =[0,1]\times [0,1]$.

For this purpose we implement a forward finite difference scheme and for this purpose initially we 
 consider a uniform partition of the domain $\Omega$, with $x=(x_1,\,x_2)\in \Omega$, $0\leq x_1\leq 1$,
$0\leq x_2\leq 1$, of $(M+1)\times (M+1)$ points with spacial step $\delta x_1=\delta x_2=\delta x$,  with
${x_1}_{\ell_1}={\ell_1} \delta x$, ${\ell_1}=0,1,\ldots M$,  ${x_2}_{\ell_2}={\ell_2} \delta x$, ${\ell_2}=0,1,\ldots M$. 

Additionally, we take a partition of $N_T$ points in the time interval $[0,T]$, where $T$ is the maximum time of the simulation, with step $\delta t$ and $t_n=n\delta t$, $i=0,\ldots N_T-1$.

Let ${U_i}_{{\ell_1},{\ell_2}}^n$ the numerical approximation of the species $i$ of the solution of equation \eqref{e1}
 at the point $({x_1}_{\ell_1},{x_2}_{\ell_2},t_n )$ of $\Omega_T=\Omega\times [0,T]$, that is $u_i({x_1}_{\ell_1},{x_2}_{\ell_2},t_n )\simeq {U_i}_{{\ell_1},{\ell_2}}^n$.
Moreover we denote by ${\mathrm{D}_i}_{{\ell_1},{\ell_2}}^n$ the corresponding approximation of the diffusion coefficients $D_{ijk}({x_1}_{\ell_1},{x_2}_{\ell_2},t_n )\simeq { \mathrm{D}_i}_{{\ell_1},{\ell_2}}^n$ and 
similarly by ${V_i}_{{\ell_1},{\ell_2}}^n$ the  approximation for the species $v$,
$v({x_1}_{\ell_1},{x_2}_{\ell_2},t_n )\simeq {V}_{{\ell_1},{\ell_2}}^n$.

\paragraph{\bf Finite difference scheme for the model equations.}

Initially we focus on the appropriate  discretization of the terms in \eqref{e1}.
For the spatial derivatives 
$\frac{\partial }{\partial x_s}\left( D_i(x,t)\frac{\partial u_i }{\partial x_s}\right)$, 
where $s=1,2$ we apply a discretization of the form 
\begin{eqnarray*}
&\hspace{-.5cm}\frac{\partial }{\partial x_1}\left( {D_i}(x,t)\frac{\partial u_i}{\partial x_1}\right)\simeq
\mathtt{\Delta} \left(u_i( {D_i} {u_i}_{x_1})\right)_{x_1}:=
\frac{1}{\delta x}\left[  {\mathrm{D}_i}_{{\ell_1}+\frac12,{\ell_2}}^n \left( \frac{{U_i}_{{\ell_1}+1,{\ell_2}}^n - {U_i}_{{\ell_1},{\ell_2}}^n }{\delta x}\right) 
- {\mathrm{D}_i}_{{\ell_1}-\frac12,{\ell_2}}^n \left( \frac{{U_i}_{{\ell_1},{\ell_2}}^n- 
{U_i}_{{\ell_1}-1,{\ell_2}}^n }{\delta x}\right) 
\right]\\
&\hspace{-.5cm}\frac{\partial }{\partial x_2}\left(  {{D}_i}(x,t)\frac{\partial }{\partial x_2}\right)\simeq
\mathtt{\Delta} \left(u_i( {{D}_i} {u_i}_{x_2})\right)_{x_2}:=
\frac{1}{\delta x}\left[  {\mathrm{D}_i}_{{\ell_1},{\ell_2}+\frac12}^n \left( \frac{{U_i}_{{\ell_1},{\ell_2}+1}^n - {U_i}_{{\ell_1},{\ell_2}}^n }{\delta x}\right) 
- {\mathrm{D}_i}_{{\ell_1},{\ell_2}-\frac12}^n \left( \frac{{U_i}_{{\ell_1},{\ell_2}}^n- {U_i}_{{\ell_1},{\ell_2}-1}^n }{\delta x}\right) 
\right]\\
&\hspace{-.5cm} {\mathrm{D}_i}_{{\ell_1}+\frac12,{\ell_2}}=
\frac{ {\mathrm{D}_i}_{{\ell_1}+1,{\ell_2}}+  {\mathrm{D}_i}_{{\ell_1},{\ell_2}}}{2}, \quad 
 {\mathrm{D}_i}_{{\ell_1}-\frac12,{\ell_2}}=
 \frac{ {\mathrm{D}_i}_{{\ell_1},{\ell_2}}+  {\mathrm{D}_i}_{{\ell_1}-1,{\ell_2}}}{2}, \ \quad
 \\
&\hspace{-.5cm} {\mathrm{D}_i}_{{\ell_1},{\ell_2}+\frac12}=
\frac{ {\mathrm{D}_i}_{{\ell_1},{\ell_2}+1}+  {\mathrm{D}_i}_{{\ell_1},{\ell_2}}}{2}, \quad 
 {\mathrm{D}_i}_{{\ell_1},{\ell_2}-\frac12}=\frac{ {\mathrm{D}_i}_{{\ell_1},{\ell_2}}+  {\mathrm{D}_i}_{{\ell_1},{\ell_2}-1}}{2}. \quad 
\end{eqnarray*}

Moreover we use a standard forward in time discretization for the time derivative
 and we conclude with a finite difference scheme of the form
for the species $u_i$'s,
\begin{eqnarray*}
{U_i}_{{\ell_1},{\ell_2}}^{n+1}={U_i}_{{\ell_1},{\ell_2}}^{n}
+\delta t\, \mathtt{\Delta} \left(U_i(D {u_i}_{x_1})\right)_{x_1}
+\delta t\,\mathtt{\Delta} \left(U_i(D {u_i}_{x_1})\right)_{x_1}
+\delta t  {R_i}_{{\ell_1},{\ell_2}}^{n}
-\delta t  {F}_{{\ell_1},{\ell_2}}^{n}
\end{eqnarray*}
and for the species $v$
\begin{eqnarray*}
V_{{\ell_1},{\ell_2}}^{n+1}={V}_{{\ell_1},{\ell_2}}^{n}+\delta t \sum_{i=1}^N \alpha_i {U_i}_{{\ell_1},{\ell_2}}^{n}-\beta V_{{\ell_1},{\ell_2}}^{n},
\end{eqnarray*}
where
\begin{eqnarray*}
 {R_i}_{{\ell_1},{\ell_2}}^{n}=
 \frac12\sum_{p+q=s}\alpha_{p,q}\beta_{p,q} {U_p}_{{\ell_1},{\ell_2}}^{n} {U_p}_{{\ell_1},{\ell_2}}^{n} -
 {U_s}_{{\ell_1},{\ell_2}}^{n}\sum_{p=1}^\infty {\alpha_{s,p}}\beta_{s,p} {U_p}_{{\ell_1},{\ell_2}}^{n},
 \end{eqnarray*}
 and
\begin{eqnarray*}
 {F}_{{\ell_1},{\ell_2}}^{n}=\frac{{L}_{{\ell_1},{\ell_2}}^{n}}{A_{{\ell_1},{\ell_2}}^{n}}\left(a_i {U_i}_{{\ell_1},{\ell_2}}^{n}-\beta_i V_{{\ell_1},{\ell_2}}^{n}\right),
\end{eqnarray*}
are the approximations of the source terms at the point $({x_1}_{\ell_1},{x_2}_{\ell_2},t_n )$.

 In addition, the functions for the length $L(r)$ and for the area $A(r)$,  are approximated, for $r\leq 1/2$  by the relations:
 \begin{eqnarray*}
{L}_{{\ell_1},{\ell_2}}^{n}=2\pi r_{{\ell_1},{\ell_2}}^{n},\quad
A_{{\ell_1},{\ell_2}}^{n}=1-\pi (r_{{\ell_1},{\ell_2}}^{n})^2,\quad \mbox{(in 2D)}.
\end{eqnarray*}


Furthermore, we have  the approximate value $r_{{\ell_1},{\ell_2}}^{n}$ of the radius $r$ given by
\begin{eqnarray*}
r_{{\ell_1},{\ell_2}}^{n+1}=r_{{\ell_1},{\ell_2}}^{n}
+ \delta t\frac{1}{r_{{\ell_1},{\ell_2}}^{n}}\alpha
\left( \sum_{i=1}^N a_i {U_i}_{{\ell_1},{\ell_2}}^{n}-\beta  V_{{\ell_1},{\ell_2}}^{n}\right) {L}_{{\ell_1},{\ell_2}}^{n}.
\end{eqnarray*}

\subsection{Basic simulation output}

In the first set of simulations we consider homogeneous Neumann boundary conditions at the three edges of the square $\Omega$, namely at $x_1=0, x_1=1$ for $0\leq x_2\leq 1$ and at $x_2=1$, $0\leq x_1\leq 1$.

At the edge $x_2=0$, $0\leq x_1\leq 1$ we impose Robin boundary conditions given by equation \eqref{mod1b}. That is  we consider a scenario  of having inflow at this side of $\Omega$ for a particular time period, $[0,t_0]$ which stops after some time $t_0$, and we want mainly to observe the deposition process of the colloid species around the solid cores of the cells. The later can be apparent by the variation in time of the radius $r$.  

We take zero distributions as initial conditions ($t=0$) for the colloidal populations, while we consider various specific initial distributions for the radius $r$. 

We consider $N=3$  mobile species $u_i$ and one immobile species $v$. Our model needs a quite large number of parameters. We take them as follows: 
$\kappa=1,\,\, 
(d_1,\,d_2,\,d_3)=(.3,.5,.99)$,\,\,
$(a_1,\,a_2,\,a_3)=(.9,.5,.3)$,\,\,
$(\beta_1,\,\beta_2,\,\beta_3)=(1,1,1)$, 
$\alpha_{i,j}=.1,\,\,\beta_{i,j}=100$, $i,j=1,\ldots 3$, $u_a^i(x)=0,\,\,v_a(x)=0,\,\, r_a(x)=.05 ,\,\,
0\leq x\leq 1$.  

Regarding the choice of boundary condition at $(x_1,0)$, we take the function $u_i^b$ to be defined as 
\[(u_1^b,\,u_2^b,\,u_3^b)=({u_1^b}_0 x_1(i-x_1),0,0)\]
with ${u_1^b}_0=25$ for $t\in [0,t_0]$ and zero for $t>t_0$, with $t_0=2$.
Moreover, we let  $b_r=0.5$, $v(x_1,x_2,0)=0$, and  $r(x_1,x_2,0)=0.1$.

In addition, we take as final  simulation time $T=3$ and set the remaining parameters to be $M=41$, $\mathrm{R}:=\delta t/\delta x^2=0.2$.

\paragraph{Approximated concentration profiles.}


In the first of the following graphs, i.e. in Figure \ref{Figu1}, concentration profiles of the colloidal population $u_1$ are plotted against space. Similar profiles are exhibited by the other colloidal populations as well. As general rule, we keep the discussion about what happens with $u_1$ only as here the effects are more visible. This corresponds also to the physical situation when most of the mass is contained in the monomer population, while the amount of observable dimer, trimer, 4-mer populations is considerably lower; see e.g. \cite{Krehel} and references cited therein.

In the first two frames we have $t<t_0$; hence we can see that there is an inflow in $\Omega$ through  one edge and so we can observe the diffusion of $u_1$ taking place in the   $x_2$ direction. In the last two frames  taken at times after $t_0$ (hence here the inflow has stopped) we see that  the concentration of $u_1$ near the edge drops possibly due to an activation of the reaction mechanisms. Especially, the deposition activates and consumes monomers initially involved in diffusion. 
\begin{figure}[htb]\hspace{9cm}
\includegraphics[bb= 330 230 250 600, scale=.9]{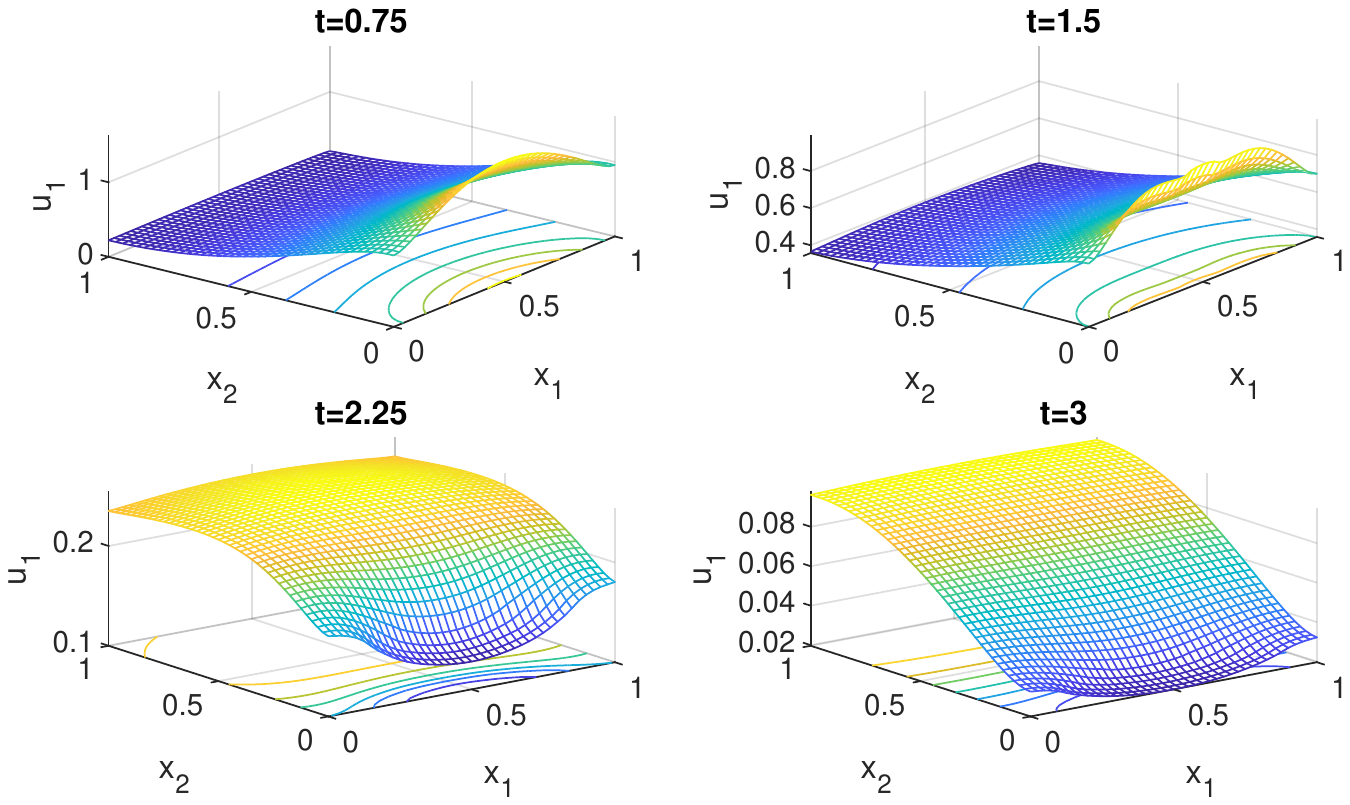}
\vspace{-2cm}
 \caption{Concentration profiles at different time steps for  the species $u_1$.}
 \label{Figu1}
\end{figure} 

In Figure \ref{Figu2}, we present similar graph for the concentration of  $u_2$.
As expected, the behaviour is similar as for the species $u_1$. Moreover, for the third species $u_3$ during the simulation we notice no difference in its qualitative behaviour.
\begin{figure}[htb]\hspace{9cm}
\includegraphics[bb= 330 230 250 600, scale=.9]{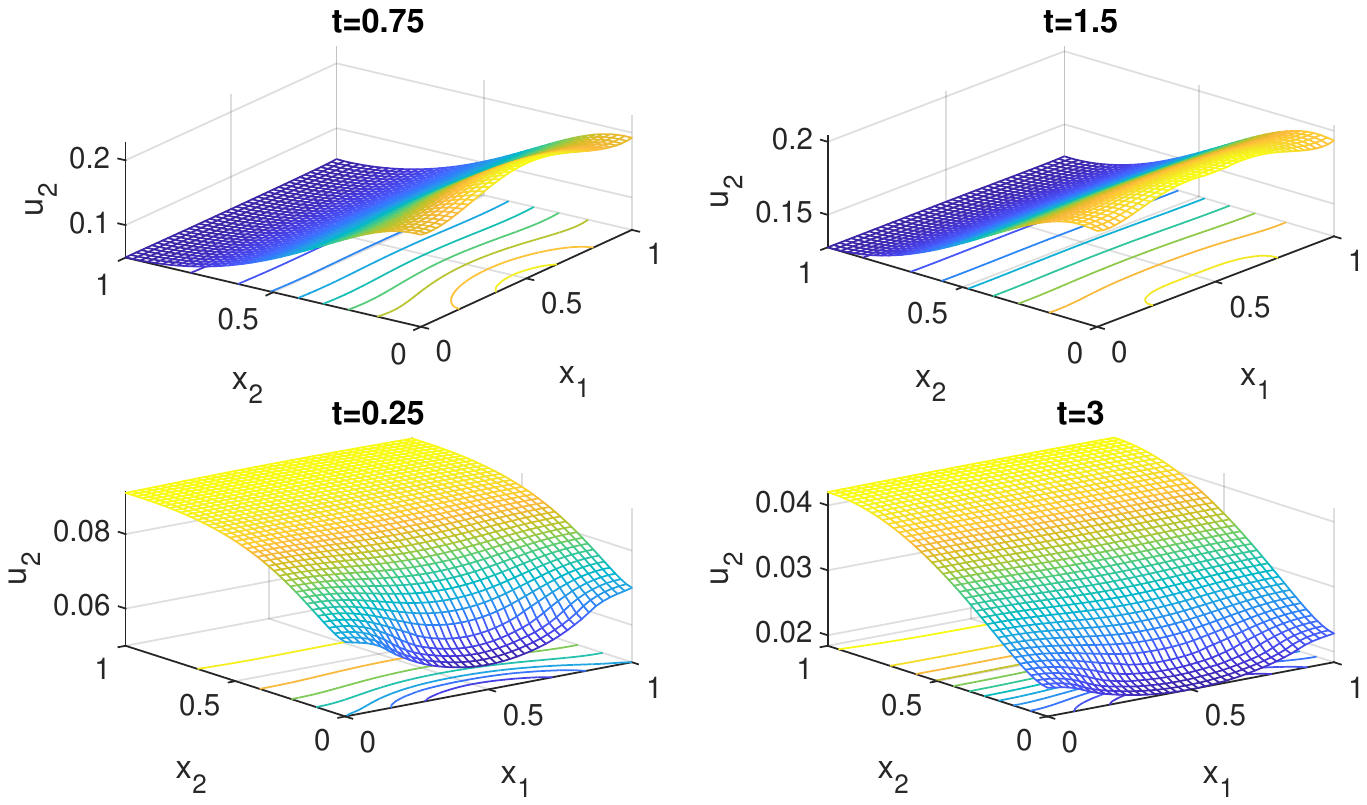}
\vspace{-2cm}
 \caption{Concentration profiles at different time steps for the species $u_2$.}
 \label{Figu2}
\end{figure} 

Regarding the behaviour of the immobile species $v$ pointed out in Figure \ref{Figv}, we observe an initial distribution in the first two frames 
$t=0.5,\, t=1.5$,  following the form of the mobile species $u_i$ and an increase inside the domain $\Omega$. After the inflow stops, for instance, see the last two frames  $t=1,75,\, t=3$,  the distribution of the mass of the deposited species appears to be stationary.
\begin{figure}[htb]\hspace{9cm}
\includegraphics[bb= 330 230 250 600, scale=.9]{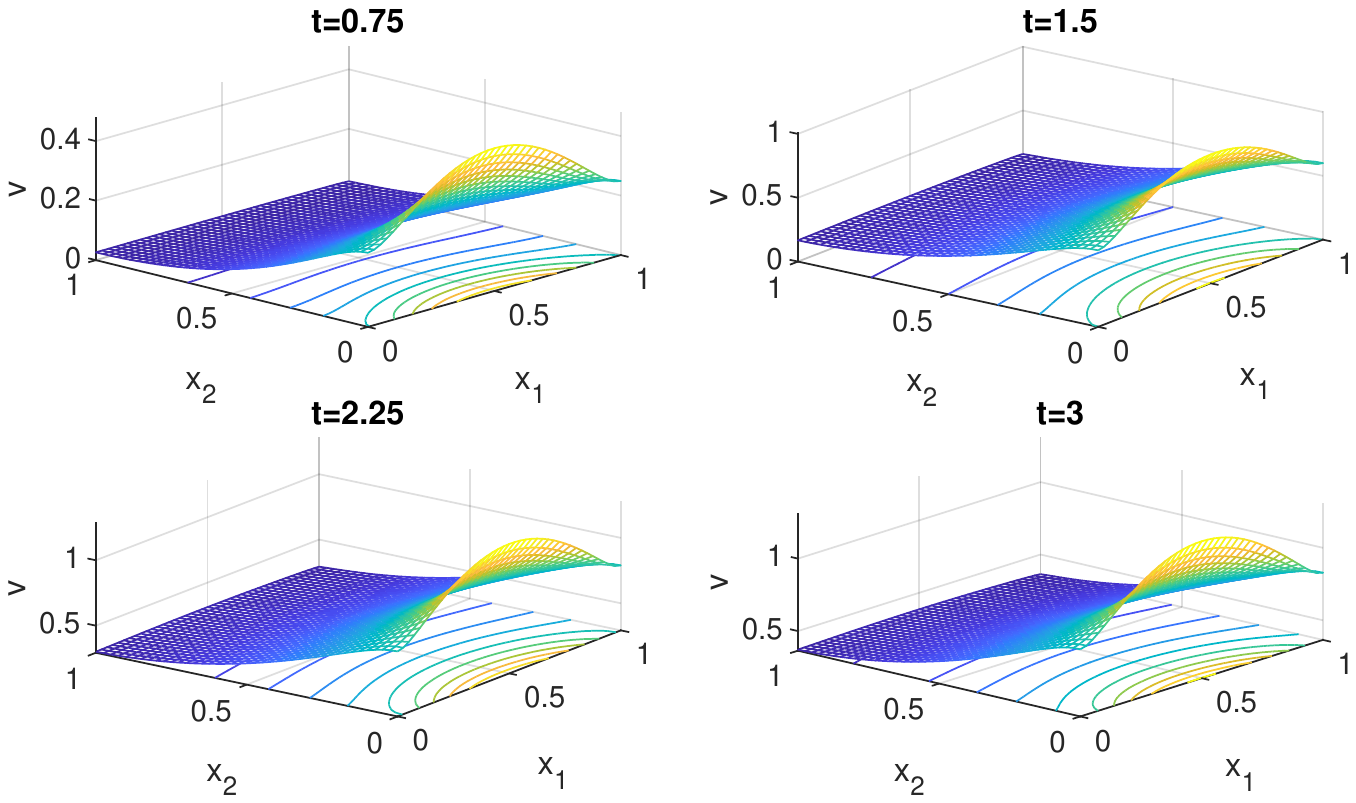}
\vspace{-2cm}
 \caption{Mass at different time steps for the deposited species $v$.}
 \label{Figv}
\end{figure} 

Focusing now in the behaviour of $r$,  we present   in  Figure \ref{Figrf}  time frames of contour plots of the 
radius at times $t_i=0.75,\,,1,5\,,2,25\,,3$. We observe the expected increase of the radius with respect to time. Even for $t>t_0=2$, after the inflow has stopped to happen,  we still have a slight increase  of the radius due to the accumulation of the immobile species around the spherical cores of the cells. 

\begin{figure}[htb]\hspace{9cm}
\includegraphics[bb= 330 230 250 600, scale=.9]{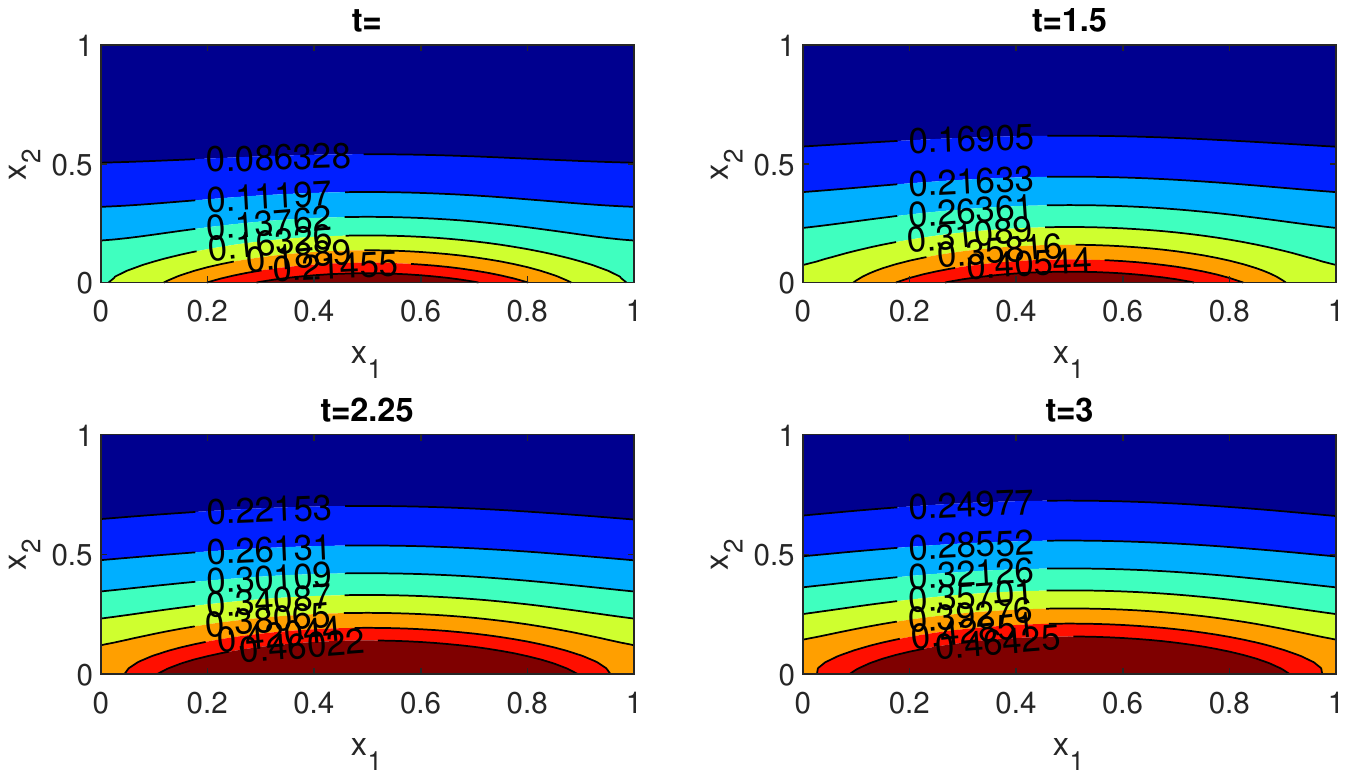}
  \vspace{-2cm}
 \caption{Contour plots of the radius $r=r(x_1,x_2,t_i)$ for the time steps $t_i=0.75,\,,1,5\,,2,25\,,3$. }
 \label{Figrf}
\end{figure} 

As final remarks regarding this numerical experiment,   the main observables $u_1$, $u_2$, $u_3$, and $v$ are plotted in Figure \ref{Figuiv_p}  against time for fixed locations  inside  the domain $\Omega$; see specifically the points $(0,0.5)$, the center $(0.5,0.5)$,  $(0.5,1)$ and at the corner $(0,0)$. 
%
\begin{figure}[htb]\hspace{9cm}
\includegraphics[bb= 330 230 250 600, scale=.9]{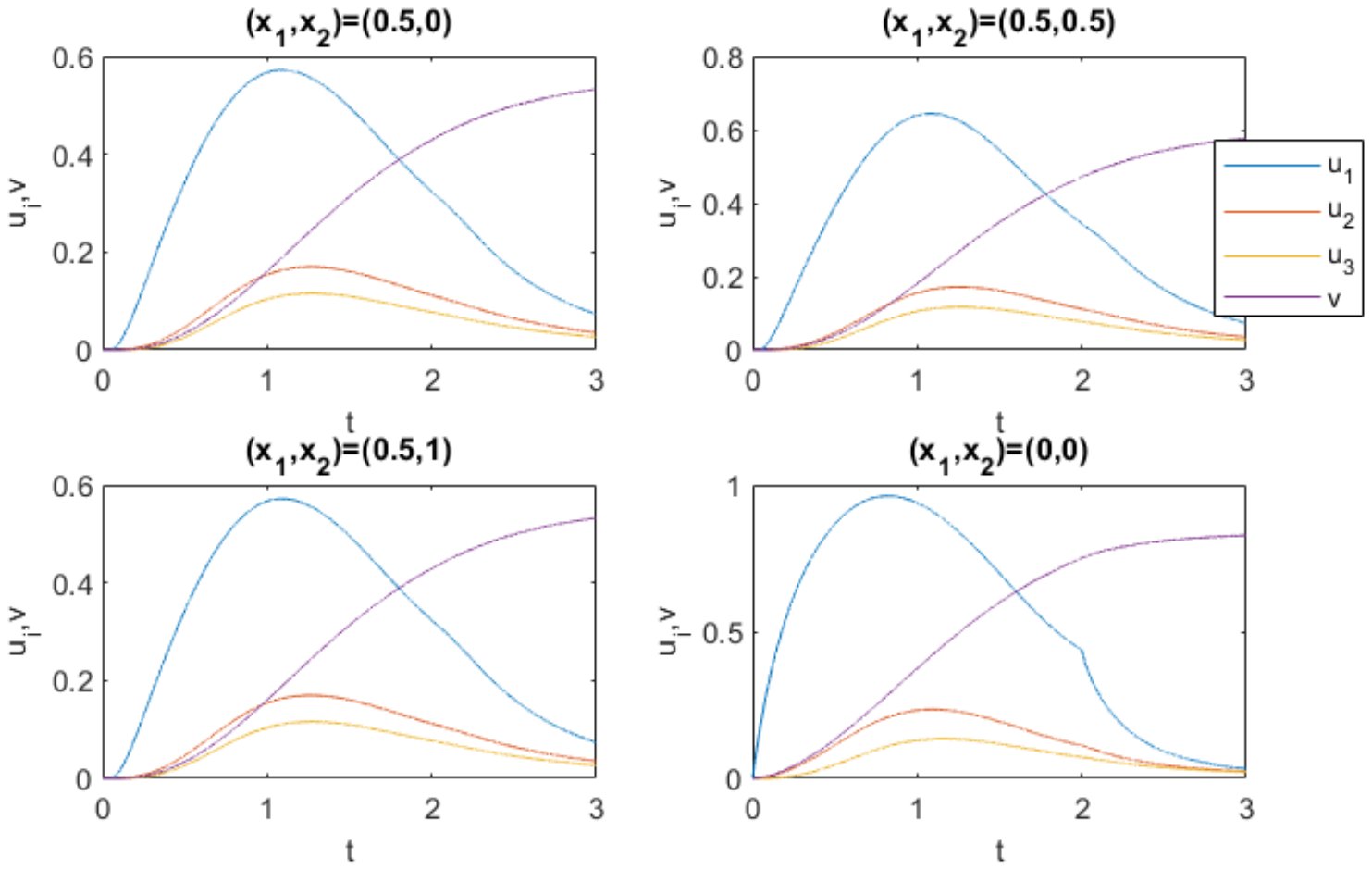}
 \caption{Concentration profiles of the species $u_i,\,v$ versus time at different spatial points in the square domain.}
 \label{Figuiv_p}
\end{figure} 

\paragraph{\bf Approximations with non uniform initial radius.}
In the following experiment we consider for the same scenario of initial and boundary conditions, \eqref{mod1b},  \eqref{mod1b2}, \eqref{mod1c},
a non uniform distribution for the initial values of the radius $r_0=r(x_1,x_2,0)$. Specifically,  we consider 
larger values of the radius in the form of two peaks centered at the points $(0.2,\,0.2)$ and $(0.8,\,0.8)$
and with $r_a$ having the form
\begin{eqnarray*}
r_a= r_c +r_1 \exp\left[-c (x_1-.2)^2-c(x_2-.2)^2\right]
+r_1 \exp\left[-c (x_1-.8)^2-c(x_2-.8)^2\right].
\end{eqnarray*}

In this context, we take  $r_c=0.05$, $r_1=0.35$, $c=60$ so that the maximum radius at these two points is quite large but smaller  than one ($\max r(x_1,x_2,0)\simeq 0.42 $) as it can be seen in the yellow area shown in Figure \ref{Figex2r0}.
Here we also set $M=41$ for the spatial partition and $\mathrm{R}=0.25$  The rest of the parameters values   are the same as in the previous numerical experiment.

The effect of the non-uniform  initial radius distribution is apparent in the evolution of the species of the model; particularly, this  non-uniformity effect can be traced back in the evolution of the population $u_1$ as exhibited in  Figure \ref{Figex2u1}. 

Due to the inflow from the edge $x_2=0$, we have now high values in the $u_1$ concentration around this edge (yellow area) of the domain, while inside the domain we have lower value (blue areas); this behavior can be seen in the first two frames of the simulation ($t=0.75,\,t=1.5$). We notice a gradual increasing  perturbation of the symmetric form of $u_1$ around the point $(0.2,\,0.2)$ due to the fact that, precisely at this point,  we have large values of $r$. In the next frames,  
at $(t=2.25,\, t=3)$ and particularly at $t=2.25$, we observe the concentration of $u_1$ after the time that the inflow in the domain has stopped ($t>t_0$ and $\frac{\partial u_i}{\partial n}(x_1,0,t)+ b_r u_i (x_1,0,t)=0$). The dominant  mechanisms now are the diffusion and the surface reaction,  i.e. the deposition of material around the cores of the cells. Thus we observe lower values of $u_1$
(blue and green areas) around the points with larger $r$ (close to the two initial peaks of $r$) where there the material has been deposited and higher  values (yellow areas) in  between the aforementioned peak points where the values of $r$
are smaller and deposition is slower. Essentially due to the same mechanism, at the final frame $t=3$ at the end of the simulation,  the values of $u_1$ decrease and  tend to zero with slower speed within the area close to the corner $(0,1)$. 
\begin{figure}[htb]\hspace{9cm}
\includegraphics[bb= 330 230 250 600, scale=.9]{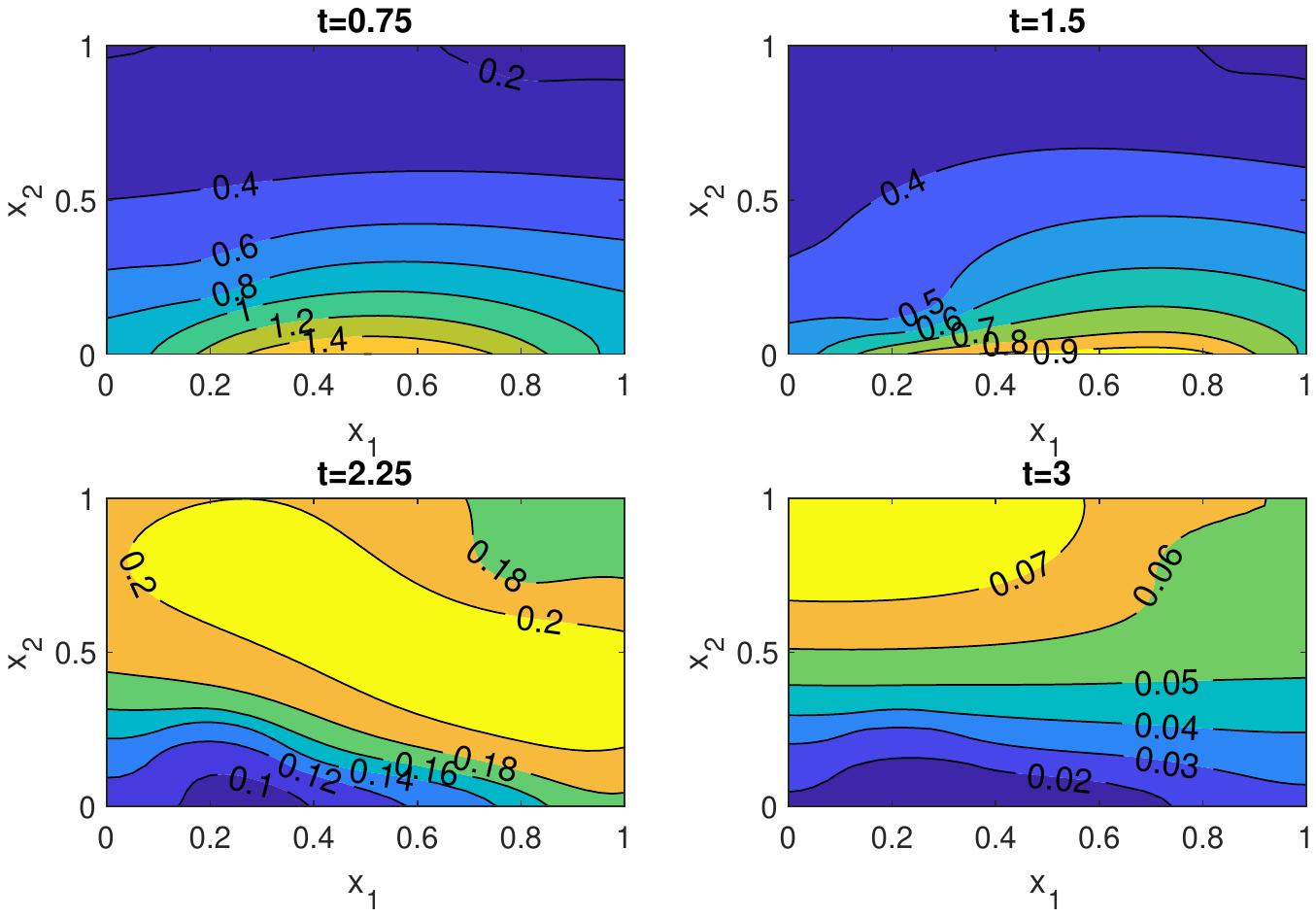}
\vspace{-2cm}
 \caption{Contour plots  at different time steps for the concentration of the species $u_1$ for the case of nonuniform initial radius distribution.}
 \label{Figex2u1}
\end{figure} 

In Figure \ref{Figex2r0}, we present the contour plot of the initial value of $r$ for this experiment. 
\begin{figure}[htb]\hspace{9cm}
\includegraphics[bb= 330 230 250 600, scale=.9]{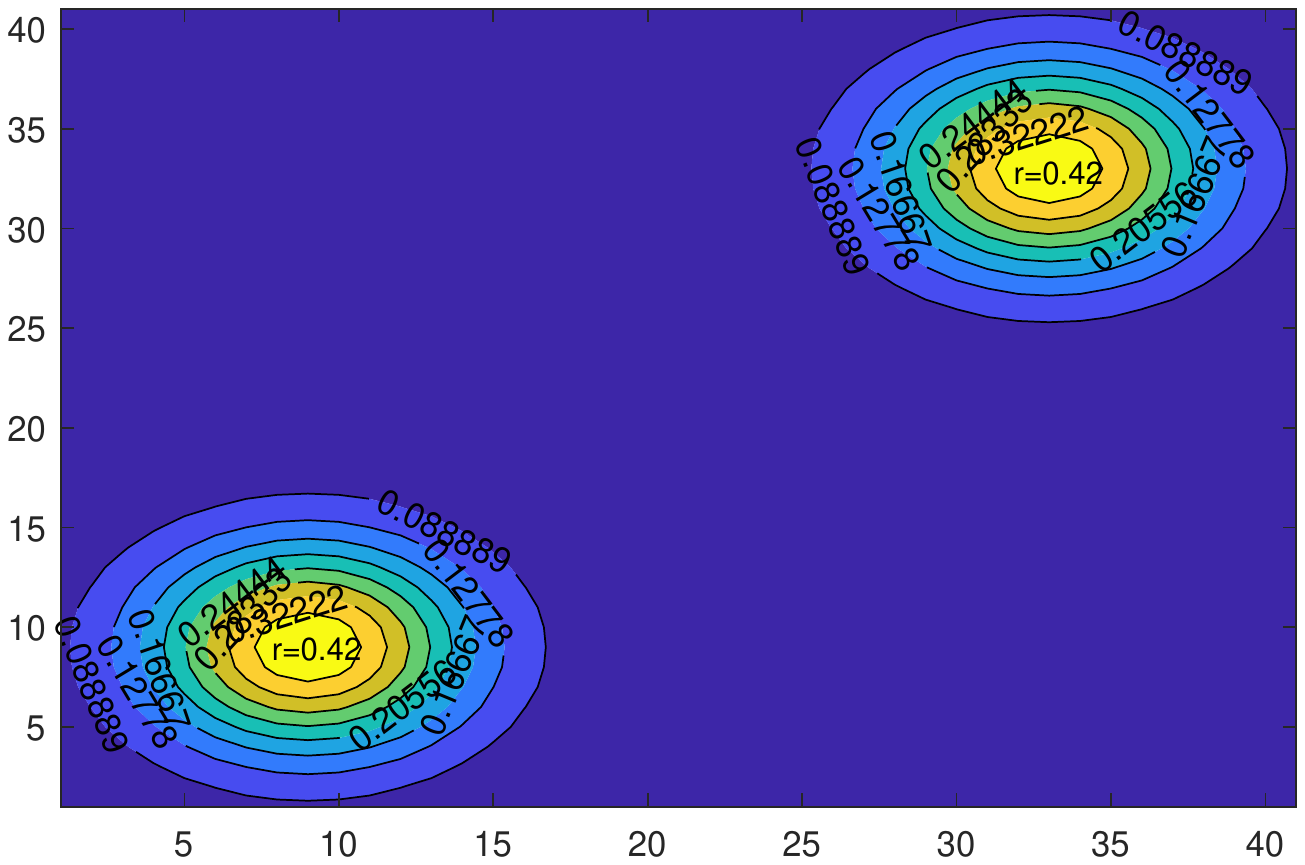}
\vspace{-2cm}
 \caption{ Contour plot   steps for  initial radius distribution $r_a=r(x_1,x_2,0)$.}
 \label{Figex2r0}
\end{figure} 
In Figure \ref{Figex2rT},  we point out  the spatial  distribution of the radius $r=r(x_1,x_2,T)$, where $T$ is the  final time of the simulation. In this case, we observe a behaviour consistent with what happens with the profile of the colloidal population $u_1$ towards the end of the simulation, i.e. around $t=3$. This effect is shown in Figure \ref{Figex2u1}. 

Higher values of $r$ equal  to $0.5$, where clogging occurs,  are taken in the lower part of the domain near the edge $x_2=0$ as well as in the neighbour of the points $(0.2,\,0.2)$ and $(0.8,\,0.8)$; 
observe the yellow areas in Figure \ref{Figex2rT}.
In the rest of the domain $\Omega$  the radius $r$ attains lower values.  This is in line  with the observed behaviour of the concentration profiles of $u_1$ around the end of the simulation. 

\begin{figure}[htb]\hspace{9cm}
\includegraphics[bb= 330 230 250 600, scale=.9]{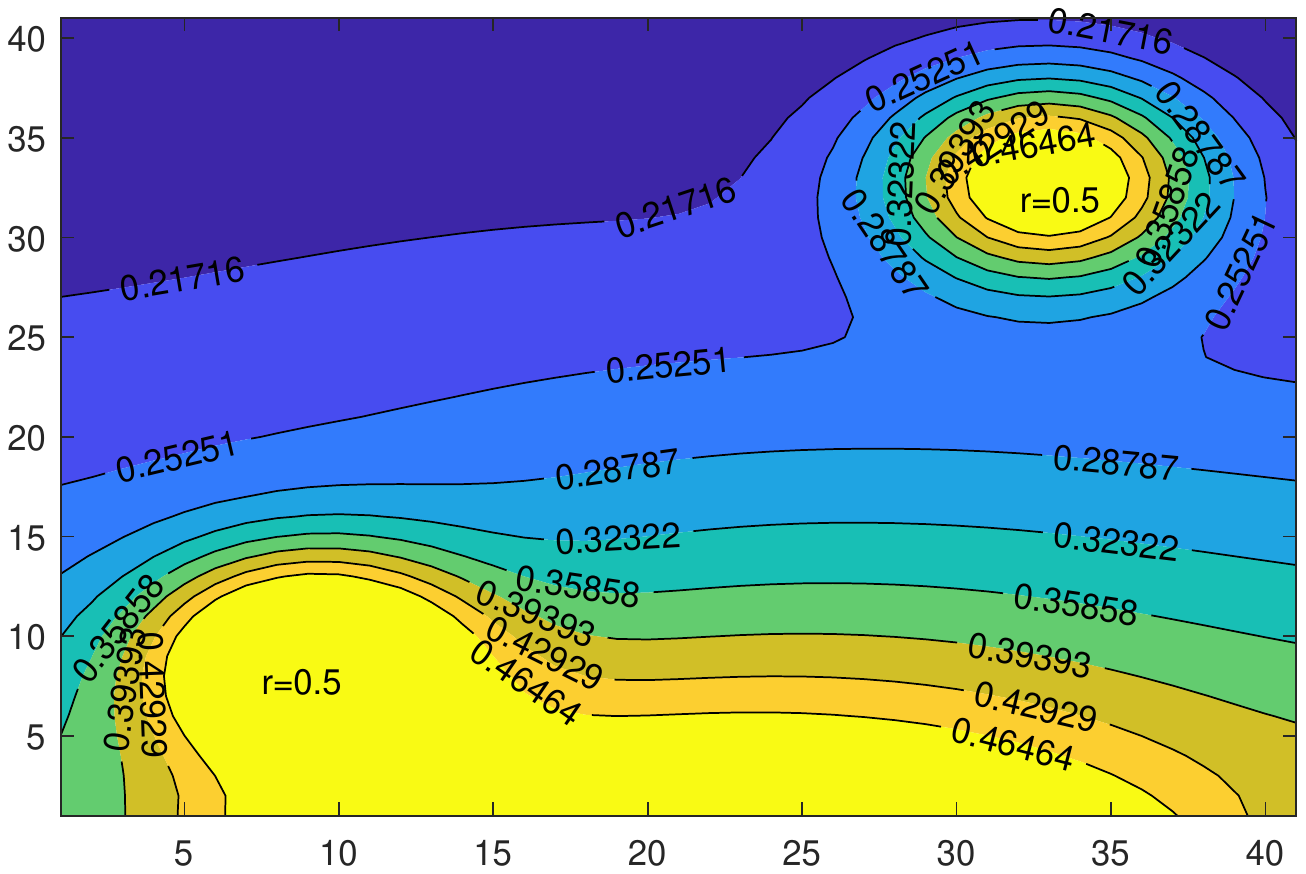}
\vspace{-2cm}
 \caption{Contour plot   for  the radius distribution $r_a=r(x_1,x_2,T)$ at the end of the simulation.}
 \label{Figex2rT}
\end{figure} 
The evolution of the diffusivity during the experiment is also apparent in Figure \ref{Figex2D}. We notice initially low values of it in the areas (blue regions) around the two peaks and higher values in the intermediate area (yellow region), in the first frame for $t=0.75$. As $r$ gradually increases the corresponding areas with low diffusivity expand as we can see in the second and third frame for $t=1.5,\, 2.25$, and finally, also for $t=3$ at the end of the simulation where we obtain the final 
map of the diffusivity. This  contains also information on the tortuosity of the material.  The latter frame  is in fact a "reverse" image of Figure
 \ref{Figex2rT} as very low values of $D$ are linked to  clogging around the blue areas where $r$ is large.
\begin{figure}[htb]\hspace{9cm}
\includegraphics[bb= 330 230 250 600, scale=.9]{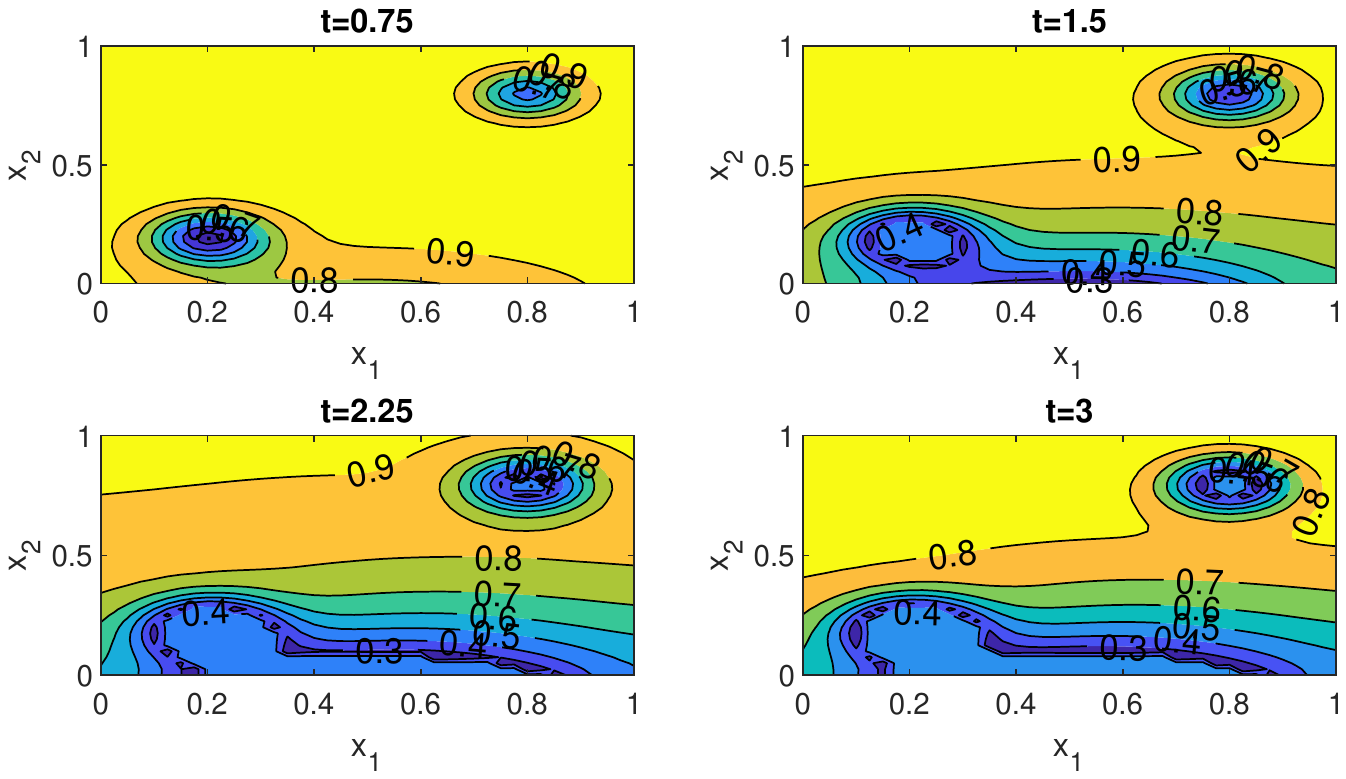}
\vspace{-2cm}
 \caption{Contour plots  at different time steps for the effective diffusivity $D(x,t)$ for the case of nonuniform initial radius distribution.}
 \label{Figex2D}
\end{figure} 

It is worthwhile to note that the spatial distribution of the balls-like microstructure  that corresponds to the vizualization shown in Figure
 \ref{Figex2rT} of the effective transport coefficient is pointed out in Figure \ref{Rdunit}. The unavoidable occurrence of clogging is pointed out in all these representations.




\section{Discussion}\label{discussion}
We have proven the existence of a weak solution to a specific coupled multiscale quasilinear system describing the diffusion, aggregation, fragmentation, and deposition of populations of colloidal particles in porous media. The structure of the system was originally derived in 
\cite{MC20} and we kept it here.

Tracking numerically the $x$-dependence in the shape of the microstructures rises serious computational problems especially in 3D or even in 2D when working with low-regular shapes. Because of the strong separation between the macroscopic length scale and the microscopic length scale, such setting is parallelizable; see \cite{Omar} for a prestudy in this direction done for a micro-macro reaction-diffusion problem with $x$-dependent microstructure arising in the context of transport of  nutrients in plants. The approach used in \cite{Omar} is potentially applicable here as well. Moreover, what concerns the discretization techniques used in this framework, a more advanced finite difference scheme, such as an appropriate version of Du Fort Frankel scheme, can give in principle more flexibility and accuracy in the numerical computations, e.g. by allowing larger time steps.  

Our multiscale model can allow for further relevant extensions in at least twofold direction: 

(1) For instance, a particularly interesting development would be to allow for some amount of stochasticity in the balance laws. In this spirit, the ODE for the growth  of the balls induced by the deposition of the species $v$  could have not only a random distribution of initial positions\footnote{This is tractable with the current form of the model.} but also  some suitably scaled "Brownian noise" in the production term mimicking an additional contribution eventually due to a non-uniform deposition of colloids on the boundary of the microstructures (compare with the setting from \cite{Maris}). The difficulty in this case is that, due to the strong coupling in the system, the overall problem becomes a quasilinear SPDE, which is much more difficult to handle mathematically and from the simulation point of view compared with our current purely deterministic setting. 

(2) Another development that would be interesting to follow in the deterministic setup is to attempt a computational efficient hybrid-type modeling. In this context, one idea  would be to couple continuum  population models for colloidal dynamics with discrete network models describing the mechanics of the underlying material (see e.g. the approach proposed in \cite{Axel} having paper as target material). Relevant questions would be: What is the counterpart of our equation for the radius growth of a ball $B(r)$, when the ball is replaced by a point? How does "continuum" deposition take place on "discrete" fixed locations?  Are points able to absorb matter in $2D$ and $3D$?  

We expect that the non-standard type of couplings suggested in (1) and (2) (i.e. deterministic-stochastic and continuum-discrete) can potentially be posed in terms of  measured-valued balance equations. We will investigate some of these ideas in follow-up works.
\clearpage
 
\section*{Acknowledgments}
AM is partially supported by the grant VR 2018-03648 "{\em Homogenization and dimension reduction of thin heterogeneous layers}". We thank R. E. Showalter (Oregon) and O. Richardson (Karlstad) for useful discussions on closely related topics.

\bibliography{colloids_idea}

\begin{thebibliography}{10}

\bibitem{Aldous}
D.~J. Aldous.
\newblock Deterministic and stochastic models for coalescence (aggregation and
  coagulation): {A} review of the mean-field theory for probabilists.
\newblock {\em Bernoulli}, pages 3--48, 1999.

\bibitem{Alt}
H.~W. Alt and S.~Luckhaus.
\newblock Quasilinear elliptic-parabolic equations.
\newblock {\em Math. Z.}, 183:311--341, 1983.

\bibitem{Maris}
H.~Bessaih, Y.~Efendiev, and R.~F. Maris.
\newblock Stochastic homogenization of a diffusion-reaction model.
\newblock {\em DCDS - Series A}, 39(9), 2019.

\bibitem{Icardi}
G.~Boccardo, E.~Crevacore, R.~Sethi, and M.~Icardi.
\newblock A robust upscaling of the effective particle deposition rate in
  porous media.
\newblock {\em Journal of Contaminant Hydrology}, 212:3--13, 2018.

\bibitem{Giulia}
G.~Bonacucina, M.~Cespi, M.~Misici-Falzi, and G.~F. Palmieri.
\newblock Colloidal soft matter as drug delivery system.
\newblock {\em Journal of Pharmaceutical Sciences}, 89(1):1--42, 2009.

\bibitem{Chadam}
J.~Chadam and P.~Ortoleva.
\newblock A mathematical problem in geochemistry: {T}he reaction-infiltration
  instability.
\newblock {\em Rocky Mountain J. Math.}, 21(2), 1991.

\bibitem{Chen}
Y.~Chen, J.~Ma, X.~Wu, L.~Weng, and Y.~Li.
\newblock Sedimentation and transport of different soil colloids: {E}ffects of
  {G}oethite and humic acid.
\newblock {\em Water}, 12:980, 2020.

\bibitem{Diaz}
C.~Conca, J.~I. Diaz, and C.~Timofte.
\newblock On the homnogenization of a transmission problem arising in
  chemistry.
\newblock {\em Romanian Reports in Physics}, 56(4):613--622, 2004.

\bibitem{King}
M.~P. Dalwadi, Y.~Wang, J.~R. King, and N.~P. Minton.
\newblock Upscaling diffusion through first-order volumetric sinks: {A}
  homogenization of bacterial nutrient uptake.
\newblock {\em SIAM J. Appl. Math.}, 78:1300--1329, 2018.

\bibitem{degenerate}
E.~DiBenedetto.
\newblock {\em {D}egenerate {P}arabolic {E}quations}.
\newblock Springer Verlag, Berlin, 1993.

\bibitem{Eden19}
M.~Eden.
\newblock Homogenization of a moving boundary problem with prescribed normal
  velocity.
\newblock {\em Adv. Math. Sci. Appl}, 28(2):313–341, 2019.

\bibitem{Fasano}
A.~Fasano and A.~Mikeli\'c.
\newblock On the filtration through porous media with partially soluble
  permeable grains.
\newblock {\em Nonlinear differ. equ. appl.}, 7:91--105, 2000.

\bibitem{Silvia}
B.~Franchi, M.~Heida, and S.~Lorenzani.
\newblock A mathematical model for {A}lzheimer’s disease: An approach via
  stochastic homogenization of the {S}moluchowski equation.
\newblock {\em Communications in Mathematical Sciences}, 18(4):1105--1134,
  2020.

\bibitem{Hallak}
B.~Hallak, E.~Specht, F.~Herz, R.~Gröpler, and G.~Warnecke.
\newblock Influence of particle size distribution on the limestone
  decomposition in single shaft kilns.
\newblock {\em Energy Procedia}, 2017.

\bibitem{Robin}
R.~J{\"a}ger.
\newblock Erosion and deposition in porous media.
\newblock Master's thesis, ETH Z\"urich, Switzerland, 2020.

\bibitem{johnson1995dynamics}
P.~R. Johnson and M.~Elimelech.
\newblock Dynamics of colloid deposition in porous media: Blocking based on
  random sequential adsorption.
\newblock {\em Langmuir}, 11(3):801--812, 1995.

\bibitem{Axel}
G.~Kettil, A.~M{\aa }lqvist, A.~Mark, M.~Fredlund, K.~Wester, and F.~Edelvik.
\newblock Numerical upscaling of discrete network models.
\newblock {\em BIT Numerical Mathematics}, 60:67--92, 2020.

\bibitem{Krehel}
O.~Krehel, A.~Muntean, and P.~Knabner.
\newblock Multiscale modeling of colloidal dynamics in porous media including
  aggregation and deposition.
\newblock {\em Advances in Water Resources}, 86:209--216, 2015.

\bibitem{Maes}
J.~Maes and C.~Soulaine.
\newblock A unified single-field volume-of-fluid-based formulation for
  multi-component interfacial transfer with local volume changes.
\newblock {\em Journal of Computational Physics}, 402:109024, 2020.

\bibitem{Meier09}
S.~A. Meier.
\newblock Global existence and uniqueness of solutions for a two-scale
  reaction-diffusion system with evolving pore geometry.
\newblock {\em Nonlinear Anal.}, 71(1-2):258--274, 2009.

\bibitem{MC20}
A.~Muntean and C.~Nikolopoulos.
\newblock Colloidal transport in locally periodic evolving porous media---{A}n
  upscaling exercise.
\newblock {\em SIAM J. Appl. Math.}, 80(1):448--475, 2020.

\bibitem{Asa}
A.~Nyfl{\" o}tt, E.~Moons, C.~Bonnerup, G.~Carlsson, L.~J{\" a}rnstr{\"o}m, and
  M.~Lestelius.
\newblock The influence of clay orientation in dispersion barrier coatings on
  oxygen permeation.
\newblock {\em Applied Clay Science}, 126:17--24, 2016.

\bibitem{PS08}
G.~Pavliotis and A.~Stuart.
\newblock {\em Multiscale {M}ethods : {A}veraging and {H}omogenization}.
\newblock Springer, New York, 2008.

\bibitem{Persson}
P.~O. Persson and G.~Strang.
\newblock A simple mesh generator in {MATLAB}.
\newblock {\em SIAM Review}, 46(2):329--345, 1998.

\bibitem{Peter09}
M.~A. Peter.
\newblock Coupled reaction-diffusion processes inducing an evolution of the
  microstructure: analysis and homogenization.
\newblock {\em Nonlinear Anal.}, 70(2):806--821, 2009.

\bibitem{Bruna_JFM}
G.~Printsypar, M.~Bruna, and I.~Griffiths.
\newblock The influence of porous-medium microstructure on filtration.
\newblock {\em Journal of Fluid Mechanics}, 86:484--516, 2019.

\bibitem{Ray}
N.~Ray.
\newblock {\em Colloidal Transport in Porous Media-Modeling and Analysis}.
\newblock PhD thesis, University of Erlangen, Germany, 2013.

\bibitem{Knabner}
N.~Ray, A.~Rupp, R.~Schultz, and P.~Knabner.
\newblock Old and new approaches predicting the diffusion in porous media.
\newblock {\em Transport in Porous Media}, 124:803--824, 2018.

\bibitem{Omar}
O.~M. Richardson, O.~Lakkis, A.~Muntean, and C.~Venkataraman.
\newblock Parallel two-scale finite element implementation of a system with
  varying microstructures.
\newblock Technical report, Karlstad University, Sweden, 2021.

\bibitem{Schulz}
R.~Schulz, N.~Ray, F.~Frank, H.~Mahato, and P.~Knabner.
\newblock Strong solvability up to clogging of an effective
  diffusion-precipitation model in an evolving porous medium.
\newblock {\em European Journal of Applied Mathematics}, pages 1--29, 2016.

\bibitem{Showalter_Oberwolfach}
R.~E. Showalter.
\newblock Distributed microstructure models of porous media.
\newblock In U.~Hornung, editor, {\em Flow in Porous Media}, pages 153--163.
  Oberwolfach, 1992.

\bibitem{Suciu}
N.~Suciu, F.~A. Radu, S.~Attinger, L.~Sch{\" u}ller, and P.~Knabner.
\newblock A {F}okker-{P}lanck approach for probability distributions of species
  concentrations transported in heterogeneous media.
\newblock {\em Journal of Computational and Applied Mathematics}, 289:241--252,
  2015.

\bibitem{Thrombosis}
C.~Valladolid, M.~Martinez-Vargas, N.~Sekhar, F.~Lam, C.~Brown, T.~Palzkill,
  A.~Tischer, M.~Auton, K.~V. Vijayan, R.~E. Rumbaut, T.~C. Nguyen, and M.~A.
  Cruz.
\newblock Modulating the rate of fibrin formation and clot structure attenuates
  microvascular thrombosis in systemic inflammation.
\newblock {\em Blood Advances}, 4(7):1340--1349, 2020.

\bibitem{Noorden}
T.~L. {van Noorden} and A.~Muntean.
\newblock Homogenisation of a locally periodic medium with areas of low and
  high diffusivity.
\newblock {\em European J. Appl. Math.}, 22:493--516, 2011.

\end{thebibliography}
\bibliographystyle{abbrv}

\end{document}